\documentclass[12pt]{article}
\usepackage[cp1250]{inputenc}
\usepackage[T1]{fontenc}
\usepackage[reqno]{amsmath}
\usepackage{amssymb}
\usepackage{amsthm}
\usepackage{hyperref}
\usepackage{graphics}
\usepackage{enumitem} 

\usepackage{blindtext}
\usepackage[affil-it]{authblk}

\usepackage{color}
\usepackage{bbm}
\usepackage[margin=1in]{geometry}
\numberwithin{equation}{section}

\newtheorem{prop}{Proposition}[section]
\newtheorem{thm}{Theorem}[section]

\newtheorem{lem}{Lemma}[section]
\theoremstyle{definition}
\newtheorem{rem}{Remark}[section]
\newtheorem{df}{Definition}[section]
\newtheorem{ex}{Example}[section]

\newcommand{\ps}{\widehat{P}^{\star}}
\newcommand{\prs}{\widehat{\pr}^{\star}}
\newcommand{\ers}{\widehat{\ew}^{\star}}
\newcommand{\cs}{\widehat{\Phi}^{\star}}
\newcommand{\supp}{\operatorname{supp}}
\newcommand{\p}{J}
\newcommand{\wxi}{\widetilde{\xi}}
\newcommand{\wPsi}{\widetilde{\Psi}}
\newcommand{\otau}{\widetilde{\tau}}
\newcommand{\n}{\mathbb{N}}
\newcommand {\<}{\left\langle}  
\renewcommand {\>}{\right\rangle}  
\newcommand {\norma}[1]{\left\|#1\right\|}
\newcommand{\pr}{\mathbb{P}}
\newcommand{\ew}{\mathbb{E}}
\newcommand{\dd}{\mathbf{d}}
\newcommand{\Lip}{\operatorname{Lip}}

\title{Exponential ergodicity in the bounded-Lipschitz distance for a subclass of piecewise-deterministic Markov processes with random switching between flows}

\newcommand\CoAuthorMark{\footnotemark[\arabic{footnote}]}
\author[1]{Dawid Czapla\protect\CoAuthorMark
  \thanks{Corresponding author;\; e-mail address: \texttt{dawid.czapla@us.edu.pl}}}
\author[1]{Katarzyna Horbacz} 
\author[1,2]{Hanna Wojew\'odka-\'Sci\k{a}\.zko}
\affil[1]{\small\textit{Institute of Mathematics, University of Silesia in Katowice, Bankowa 14, 40-007 Katowice, Poland}} 
\affil[2]{\small\textit{Institute of Theoretical and Applied Informatics, Polish Academy of Sciences, Ba\l tycka 5, \hbox{44-100 Gliwice, Poland}}}
\date{}
\begin{document}
\maketitle
\vspace*{-1.3cm}
\begin{abstract}
In this paper, we study a subclass of piecewise-deterministic Markov processes with a Polish state space, \hbox{involving} deterministic motion punctuated by random jumps that occur at exponentially distributed time intervals. Over each of these intervals, the process follows a flow, selected randomly among a finite set of all possible ones. Our main goal is to provide a set of verifiable conditions guaranteeing the exponential ergodicity for such processes (in terms of the bounded Lipschitz distance), which would refer only to properties of the flows and the transition law of the Markov chain given by the post-jump locations. Moreover, we establish a simple criterion on the exponential \hbox{ergodicity} for a particular instance of these processes, applicable to certain biological models, where the jumps result from the action of an iterated function system with place-dependent probabilities.
\end{abstract}
{\small
}

\section{Introduction}
Piecewise-deterministic Markov processes (PDMPs), first introduced by Davis \cite{b:davis} in 1984 (see~also~\cite{b:Davis_book1,b:Davis_book2}), constitute a general class of non-diffusive Markov processes, for which randomness stems only from the jump mechanism, including the jumping times, the post-jump locations and other changes occurring at the moments of jumps. This huge family of processes is extensively used for modelling purposes in many applied subjects, like biology \cite{b:buckwar,b:des, b:crudu, b:tyran, b:rudnicki}, storage modelling \cite{b:boxma} or internet traffic \cite{b:graham}.

In this paper, we are concerned with the PDMPs that evolve on a Polish space  through jumps arriving according to a Poisson process. This means that the span of time between consecutive jumps is exponentially distributed with a constant rate~$\lambda$. Between any two adjacent jumps, the dynamics of these processes is driven by one of the semiflows, \hbox{randomly} selected from a finite set \hbox{$\{S_i:\, i\in I\}$} of possible ones, according to a given stochastic \hbox{matrix}~$[\pi_{ij}]_{i,j\in I}$. The state right after a jump (usually called the post-jump location) depends randomly on the one immediately preceding this jump, and its probability distribution is governed by a~given Markov transition function (a stochastic kernel) $(y,B)\mapsto J(y,B)$.

More specifically, given an arbitrary Polish metric space $Y$, we shall investigate a stochastic process \hbox{$\Psi:=\{(Y(t),\xi(t))\}_{t\geq 0}$} with values in $X:=Y\times I$, whose motion can be described as follows. Starting from some initial value $(y_0,i_0)$, the process evolves in a deterministic way, following  $t\mapsto S_{i_0}(t,y_0)$ until the first jump time, say $\tau_1=t_1$. At this moment the trajectory jumps to another point of $Y$, say $y_1$, so that the probability it will fall into a Borel set $B\subset Y$ is $J(S_i(t_1,y_0),B)$. At the same time the index of the "active" semiflow is randomly switched from $i_0$ to another (or the same) one $i_1$ with probability $\pi_{i_0i_1}$. Then the motion restarts from the new state $(i_1,y_1)$ and proceeds as before. Formally, the process $\Psi$ can be therefore defined by setting
$$
Y(t):=S_{\xi_n}(t-\tau_n,Y_n)\quad\text{and}\quad \xi(t):=\xi_n\quad\text{for}\quad t\in [\tau_n,\tau_{n+1}),\;\;n\in\n_0,
$$
where $\bar{\Phi}:=\{(Y_n,\xi_n,\tau_n)\}_{n\in\n\cup\{0\}}$ is a time-homogeneous Markov chain with state space \hbox{$X\times [0,\infty)$} and transition law satisfying
$$
\pr(\bar{\Phi}_{n+1}\in B\times \Xi\times T|\,\bar{\Phi}_n=(y,i,s))=\sum_{j\in \Xi}\pi_{ij} \int_{T\cap\, [s,\infty)} \lambda e^{-\lambda (t-s)} \p(S_i(t,y),B)\,dt
$$
for any $n\in\n\cup\{0\}$, $y\in Y, i\in I,s\geq 0$ and Borel sets $B\subset Y$, $\Xi\subset I$, $T\subset[0,\infty)$. Obviously, all the randomness of the PDMP $\Psi$ is contained in the chain $\bar{\Phi}$. What is more, the sequence $\Phi:=\{(Y_n,\xi_n)\}_{n}$ of the post-jump locations itself is an $X$-valued Markov chain (with respect to its natural filtration). Clearly, on the family of rectangles $B\times \Xi$ (where $B\subset Y$ is a Borel set, and $\Xi\subset I$), the transition law of this chain takes the form
$$P((y,i),B\times \Xi):=\pr(\Phi_{n+1}\in B\times \Xi\,|\,\Phi_n=(y,i))=\sum_{j\in \Xi} \pi_{ij}\int_0^{\infty} \lambda e^{-\lambda t}\p(S_i(t,y),B)\,dt.$$

The subclass of the PDMPs considered here somewhat resembles those investigated in \cite{b:benaim4, b:benaim2, b:benaim1, b:benaim3, b:costa, b:costa2}. All these papers, however, focuse on processes evolving on finite-dimensional (and thus locally compact) spaces. While proving the existence of invariant distributions and ergodicity (usually in the total variation norm) in such a setup, one can use various adaptations of conventional methods of Meyn and Tweedie \cite{b:meyn2, b:meyn1}, based mainly on the Harris recurrence (assured e.g. by H\"ormander-type bracket conditions, just as in \cite{b:benaim1}) or some criteria referring to the so-called drift towards a petite set. These techniques, however, are mostly valid only for $\psi$-irreducible processes, which is, obviously, not the case in our framework. On the other hand, \cite{b:cloez_hairer}, for instance, deals with a large class of regime switching Markov processes (a much more general family than that of PDMPs), which take values in a Polish space. Nevertheless, the criteria on the exponential ergodicity (in the Wasserstein distance) provided in that work are based on fairly general assumptions, such as the ``exponential contractivity'' of the given Markov semigroups or a Lyapunov-Foster type condition in the continuous-time context, which might be difficult to verify in practice (at least in a direct~way).

The main goal of this paper is to provide relatively easy to check conditions on the kernel~$J$ and the semiflows $S_i$ which would guarantee that both the transition operator of the chain~$\Phi$ and the transition semigroup of the process $\Psi$ are exponentially ergodic in the bounded Lipschitz distance (equivalent to the one induced by the Dudley norm \cite{b:dudley}). Such a metric, also known as the Fortet--Mourier distance (see e.g. \cite{b:las}), is defined on the cone of non-negative finite Borel measures on $X$, and induces the topology of weak convergence of such measures (\hbox{\cite{b:bogachev}}). Roughly speaking, the aforementioned form of ergodicity means that the process under consideration admits a unique stationary (invariant) distribution, to which its distribution converges at an exponential rate in the Fortet--Mourier distance, independently of the initial state. The rigorous meaning of this term is given in Definitions \ref{def:erg_p}~and~\ref{def:erg_pt}.  
The general strategy of our approach is as follows:
\begin{enumerate}[label=(\Roman*)]
\item\label{step:i} We begin with showing that, whenever $J$ enjoys some strengthened form of the Feller property, there exists a one-to-one correspondence between the set of invariant distributions of the process~$\Psi$ and those of the associated chain $\Phi$ (Theorem~\ref{thm:cor}). 

\item\label{step:ii} Next, we note that the existence of an appropriate coupling $(\Phi^{(1)},\Phi^{(2)})$ between two copies of the chain $\Phi$, such that the mean distance between them decreases geometrically with time, in conjunction with a Foster--Lyapunov condition (\hbox{\cite[Definition 6.23]{b:douc}}) and the Feller property imposed on $P$, ensures the exponential ergodicity of $\Phi$\linebreak (Lemma~\ref{lem:a}). Obviously, the latter two assumptions can be equivalently formulated with respect to $J$.

\item\label{step:iii} The essential step in our analysis is proving that, for a given coupling $(\Phi^{(1)},\Phi^{(2)})$ of the chain $\Phi$ enjoying the property indicated in \ref{step:ii}, the corresponding coupling for the process~$\Psi$ has an analogous property, provided that the semiflows~$S_i$ fulfil a certain Lipschitz-type condition (Lemma \ref{lem:b}). The key idea here is partially inspired by the techniques used in the proof of \cite[Theorem 1.4]{b:cloez_hairer}.

\item\label{step:iv} From the results discussed in steps \ref{step:i} and \ref{step:iii} we can conclude that, upon making suitable assumptions on the semiflows~$S_i$ and the kernel $J$, the existence of an appropriate coupling of $\Phi$ implies the exponential ergodicity of the process~$\Psi$ (Theorem~\ref{thm:main1}).

\item\label{step:v} Finally, we introduce some additional hypotheses which, together with the previous ones, ensure that the coupling mentioned in~\ref{step:iii} exists. This leads us to the main result of the paper, stated as Theorem \ref{thm:main2}. More precisely, at this stage we require the existence of a substochastic kernel $Q_J$ on $Y^2$ with certain specific properties (in the spirit of \cite{b:czapla_clt, b:kapica}), such that
$$Q_J((y_1,y_2),\cdot\times Y)\leq J(y_1,\cdot) \quad\text{and}\quad Q_J((y_1,y_2), Y\times \cdot)\leq J(y_2,\cdot),$$
which further enables us to construct a substochastic kernel $Q_P$ on $X^2$, having the analogous properties with respect to $P$ (Lemma \ref{lem:help}). The transition function of the desired coupling can be then defined as the sum of $Q_P$ and a suitable complementary kernel (Proposition \ref{lem:coup}). Such a construction is inspired by the ideas of Hairer \cite{b:hairer}, regarding the so-called asymptotic coupling technique (also used e.g. in \cite{b:sleczka,b:wojewodka}).
\end{enumerate}

What is especially noteworthy here is the fact that this approach also elucidates the way in which the exponential ergodicity of the PDMP $\Psi$ is inherited from the same property for the associated chain $\Phi$. This is visible in steps \ref{step:i} and \ref{step:iii}.


The obtained general result (i.e.\,Theorem \ref{thm:main2}) is further applied to derive a simple \hbox{criterion} on the exponential ergodicity (in the Fortet--Mourier distance) in the case where the jump kernel $J$ is a transition law of a random iterated function system (Proposition~\ref{prop:ifs}). This is done by taking advantage of the fact that the kernel $Q_J$, playing a key role in step~\ref{step:v}, can be defined explicitly in such a model. More specifically, we discuss the case in which $J$ is given by
$$
\p(y,B)=\int_{w_{\theta}^{-1}(B)}p_{\theta}(y)\,\vartheta(d\theta)\quad \text{for each}\;\;y\in Y\;\;\text{and any Borel set}\;\;B\subset Y,
$$
where $\{w_{\theta}:\,\theta\in\Theta\}$ is an arbitrary family of continuous transformations from $Y$ to itself, indexed by the elements of a measure space $(\Theta,\vartheta)$, and  $\{p_{\theta}:\,\theta\in \Theta\}$ is an associated set of place-dependent probabilities. In this setting, the model under consideration may serve as a~framework for analysing the dynamics of gene expression in prokaryotes (see e.g. \cite{b:tyran2, b:czapla_erg, b:tyran}). More precisely, $Y(t)$ then describes the concentration of a protein at time $t$, during its degradation process. This process is interrupted by the so-called transcriptional bursts, followed by variable periods of inactivity, with exponentially distributed duration (expressed by the increments of~$\tau_n$). The bursts can be represented by the action of randomly selected transformations of the form $w_{\theta}(y)=y+\theta$, so that the process changes from $Y(\tau_k-)$ to $Y(\tau_k)=Y(\tau_k-)+\theta_k$ with some random variable $\theta_k$, representing the size of the $k$-th burst. Clearly, in this context, the semiflows $S_i$ determine the dynamics of the degradation process between consecutive bursts, which depends on the current amount of the gene product. Such a dynamics might be also perturbed by occasional fluctuations in the environment, caused by the bursts, which is modelled by the random switching between flows.

The discrete-time dynamical system $\Phi$ with the above-specified shape of the jump \hbox{kernel}~$J$, even in a more general setting, wherein the probabilities $\pi_{ij}$ depend on the state, has been more widely examined (in terms of ergodicity and classical limit theorems) in our previous articles \cite{b:czapla_erg, b:czapla_lil, b:czapla_clt, b:czapla_erg2}. For instance, in \cite{b:czapla_erg}, the exponential ergodicity of $\Phi$ has been used to prove the strong law of large numbers for the chain $\{f(\Phi_n)\}_{n}$ (with a Lipschitz continuous function $f:X\to\mathbb{R}$), which, in turn, has enabled us to derive the analogous law for the process $\{f(\Psi(t))\}_{t\geq 0}$ (without using the ergodicity of $\Psi$). The result provided in the present paper should prove to be useful in establishing also the central limit theorem for this process, which would be rather difficult to achieve based only on the properties of $\Phi$.

The organization of the paper is as follows. In Section \ref{sec:2}, we introduce notation and some basic concepts regarding Markov semigroups acting on measures, including the \hbox{employed} \hbox{definition} of ergodicity. Section \ref{sec:3} provides a detailed description of the subclass of the PDMPs under study. Section \ref{sec:4} is devoted to establishing a one-to-one correspondence between invariant distributions of the processes $\Psi$ and $\Phi$, that is, the realization of step~\ref{step:i}. The essential part of our analysis, referring to the coupling argument, which has been described within steps \ref{step:ii}-\ref{step:iv}, is contained in Section \ref{sec:coup}. Step~\ref{step:v}, including the construction of a~suitable coupling for $\Phi$, is included in Section \ref{sec:main}. Finally, also in this part of the paper, we state the main result and discuss the special case of the model, for which the jumps are determined by a random iterated function system.



\section{Preliminaries}\label{sec:2}
Consider a complete separable metric space $(E,\rho)$, endowed with its Borel $\sigma$-field $\mathcal{B}(E)$. By~$B_E(x,r)$ we will denote the open ball in $E$ centered at $x$ of radius $r>0$. The symbol $\mathbbm{1}_A$ will be used to denote the indicator function of a subset $A$ of $E$ (or any other space, which should be clear from the context). Additionaly, we set $\mathbb{N}_0:=\mathbb{N}\cup\{0\}$ and $\mathbb{R}_+:=[0,\infty)$.

Let $B_b(E)$ stand for the Banach space of all real-valued, Borel measurable \hbox{functions} on~$E$, equipped with the supremum norm  $\norma{f}_{\infty}:=\sup_{x\in E}|f(x)|$. By $C_b(E)$ we shall denote the subspace of $B_b(E)$ consisting of all continuous functions. In addition to this, we also define the set $\Lip_{b,1}(E)$ as follows:
$$\Lip_{b,1}(E):=\left\{f\in C_b(E):\, 0\leq f \leq 1,\;\;\sup_{x\neq y}\frac{|f(x)-f(y)|}{\rho(x,y)}\leq 1\right\}.$$

Moreover, we will write $\mathcal{M}(E)$ and $\mathcal{M}_{prob}(E)$ to denote the cone of all finite non-negative, Borel measures on $E$, and its subset consisting of all probability measures, respectively. Further, given any Borel measurable function $V:E\to [0,\infty)$, we shall consider the subset $\mathcal{M}_{prob}^V(E)$ of $\mathcal{M}_{prob}(E)$ consisting of all measures with finite moment with respect to $V$, i.e.
$$\mathcal{M}_{prob}^V(E):=\left\{\mu\in\mathcal{M}_{prob}(E):\; \int_E V(x)\,\mu(dx)<\infty \right\}.$$

For brevity of notation, the Lebesgue integral $\int_E f\,d\mu$ of a Borel measurable function \linebreak\hbox{$f:E\to\mathbb{R}$} with respect to a signed Borel measure $\mu$ -- if exists -- will be sometimes denoted by~$\<f,\mu\>$. Furthermore, we will write $\delta_x$ for the Dirac measure at $x\in E$ on~$\mathcal{B}(E)$. 

To describe the distance between measures, we will use the the Fortet--Mourier metric (equivalent to the metric induced by the Dudley norm \cite{b:dudley}), which on $\mathcal{M}(E)$, is defined~by
$$d_{FM,\rho}(\mu,\nu):=\sup_{f\in \Lip_{b,1}(E)} |\<f,\mu-\nu\>|\quad\text{for any}\quad \mu,\nu\in\mathcal{M}(E).\vspace{-0.1cm}$$
It is well-known that, as long as $E$ is separable (which is the case here), the metric $d_{FM,\rho}$ induces the topology of weak \hbox{convergence} of measures on $\mathcal{M}(E)$ (cf. \cite[Theorem 8]{b:dudley} or \hbox{\cite[Theorem 8.3.2]{b:bogachev}}). Let us recall here that a sequence \hbox{$\{\mu_n\}_{n\in\n}\subset\mathcal{M}(E)$} of measures is called weakly convergent to a measure $\mu\in\mathcal{M}(E)$ if $\<f,\mu_n\>\to\<f,\mu\>$, as $n\to \infty$, for any $f\in C_b(E)$.  Moreover, if~$(E,\rho)$ is complete (which is also the case in our setting), then so is the space $(\mathcal{M}_{prob}(E),d_{FM,\rho})$ (see \cite[Theorem 9]{b:dudley}).

Before further discussion, it is also useful to recall several basic concepts in the theory of Markov operators. 

A function $P:E\times\mathcal{B}(E)\rightarrow \left[0,1\right]$ is called a \emph{(sub)stochastic kernel} if for each \hbox{$A\in\mathcal{B}(E)$}, $x\mapsto P(x,A)$ is a measurable map on $E$, and for each $x\in E$, $A\mapsto P(x,A)$ is a~(sub)probability Borel measure on $\mathcal{B}(E)$. The composition of two such kernels, say $P$ and $Q$, is defined by
\begin{equation}
\label{e:comp}
PQ(x,A):=\int_X Q(y,A)P(x,dy)\quad\text{for any}\quad x\in E,\;A\in\mathcal{B}(E).
\end{equation}
According to this rule, we can also define recursively the so-called $n$-step kernel $P^n$, by setting $P^1:=P$ and $P^{n+1}:=P^nP$ for every $n\in\n$.

For any (sub)stochastic kernel $P$, we can consider two operators (which will be denoted by the same symbol), one acting on $\mathcal{M}(E)$, and the second one acting on $B_b(E)$, defined by
\begin{gather}
\label{def:markov}\mu P(A):=\int_E P(x,A)\,\mu(dx)\quad\text{for}\quad \mu\in\mathcal{M}(E),\;A\in\mathcal{B}(E),\\
\label{def:dual}Pf(x):=\int_E f(y)\, P(x,dy)\quad\text{for}\quad f\in B_b(E),\;x\in E.
\end{gather}
Note that these operators are related to each other in the following way:
$$
\<f,\mu P\>=\<Pf,\mu\>\quad\text{for any}\quad f\in B_b(E),\;\mu\in\mathcal{M}(E).
$$
Obviously, the $n$th iterations $(\cdot)P^n$ and $P^n(\cdot)$ are induced by the $n$-step kernel $P^n$. If the kernel $P$ is stochastic, then $P:\mathcal{M}(E)\to\mathcal{M}(E)$, given by \eqref{def:markov}, is called a \emph{(regular) Markov operator}, and $P:B_b(E)\to B_b(E)$, defined by \eqref{def:dual}, is said to be its dual operator. Let~us stress that formula \eqref{def:dual} will be sometimes applied, with a slight abuse of notation, to unbounded above functions as well; for example, we shall write $P\rho(\cdot,x_0)$ (for a fixed $x_0\in E$).

A family of Markov operators $\{P_t\}_{t\in\mathbb{R}_+}$ is called a \emph{Markov semigroup} if $P_s\circ P_t=P_{s+t}$ for any $s,t\geq 0$.

By a \emph{time-homogeneous Markov chain} with (one-step) transition law $P$ and initial distribution $\mu\in\mathcal{M}_{prob}(E)$ we mean a sequence of \hbox{$E$-valued} random variables \hbox{$\Phi:=\{\Phi_n\}_{n\in\n_0}$}, defined on some probability space $(\Omega,\mathcal{F},\mathbb{P}_{\mu})$, such that, for any $A\in\mathcal{B}(E)$ and $n\in\n$,
\begin{gather}
\label{def:m_init}\pr_{\mu}(\Phi_0\in A)=\mu(A),\\
\label{def:m_chain}\pr_{\mu}(\Phi_{n+1}\in A\,|\,\mathcal{F}_n)=\pr(\Phi_{n+1}\in A\;|\;\Phi_n)=P(\Phi_n,A),
\end{gather}
where $\mathcal{F}_n$ is the $\sigma$-field generated by $\Phi_0,\ldots,\Phi_n$. The expectation operator with respect to~$\pr_{\mu}$ is then denoted by $\ew_{\mu}$. In the case where $\mu=\delta_x$ with some $x\in E$, we simply write $\pr_x$ and~$\ew_x$ rather than $\pr_{\delta_x}$ and  $\ew_{\delta_x}$, respectively. Obviously $\pr_x=\pr_{\mu}(\cdot | \Phi_0=x)$ for any $x\in E$. 

One can easily check that $P$, specified by \eqref{def:m_chain}, defines a stochastic kernel on $E\times\mathcal{B}(E)$, and that, for every $k\in\n$, the $k$-step transition probabilities of $\Phi$ are determined by the kernel $P^k$, i.e. $\pr(\Phi_{n+k}\in A\;|\;\Phi_n)=P^k(\Phi_n,A)$. Consequently, it follows that the Markov operator $(\cdot)P$ describes the evolution of the distribution of $\Phi$, i.e. $\mu_{n+1}=\mu_n P$ for any $n\in\n_0$, where $\mu_n$ is the distribution of $\Phi_n$. In this connection, it is reasonable to call~$(\cdot)P$ the \emph{transition operator} of~$\Phi$. Furthermore, it is also worth noting that the dual operator of $(\cdot)P^n$ can be expressed~as 
\begin{equation}
\label{e:dual}
P^nf(x)=\ew_x[f(\Phi_n)]\quad\text{for any}\quad x\in E,\;f\in B_b(E),\;n\in\n.
\end{equation}

On the other hand, it is well-known that, for any given stochastic kernel $P$ on $E\times\mathcal{B}(E)$ and $\mu\in\mathcal{M}_{prob}(E)$, on some probability space $(\Omega,\mathcal{F},\pr_{\mu})$, there exists a time-homogeneus Markov chain $\Phi$ with transition law $P$ and initial measure $\mu$ (see e.g. \cite{b:revuz}). In practise, it is convenient to assume that $\Omega:=E^{\mathbb{N}_0}$, $\mathcal{F}:=\mathcal{B}(E^{\mathbb{N}_0})$ (where $E^{\mathbb{N}_0}$ is endowed with the product topology), and that $\{\Phi_n\}_{n\in\n_0}$ is a sequence of canonical projections from $\Omega$ to $E$, that is, $\Phi_n(\omega)=x_n$ for any $\omega=(x_0,x_1,\ldots)\in\Omega$, with $x_0,x_1,\ldots\in E$. Then, for each $\mu\in\mathcal{M}_{prob}(E)$, one can construct a probability measure $\pr_{\mu}$ on $\mathcal{F}$ such that
\begin{equation}
\label{e:canonical}
\mathbb{P}_{\mu}(F)=\int_{E}\int_{E}\ldots \int_{E}\mathbbm{1}_{A_0\times\ldots\times A_n}(x_0,\ldots,x_n)P(x_{n-1},dx_n)\ldots P(x_0,dx_1)\mu(dx_0)
\end{equation}
for any $n\in\n_0$ and $F=\{\Phi_0\in A_0,\ldots,\Phi_n\in A_n\}$, where $A_0,\ldots,A_n\in\mathcal{B}(E)$. It then follows easily that $\Phi$ obeys \eqref{def:m_init} and \eqref{def:m_chain} for every $\mu\in\mathcal{M}_{prob}(E)$, and, what is more, we have
\begin{equation}
\label{e:p_m}
\pr_{\mu}(F)=\int_X \pr_x(F)\,\mu(dx)\quad\text{for any}\quad F\in\mathcal{F},\;\mu\in\mathcal{M}_{prob}(E).
\end{equation}
The Markov chain constructed in this way is called a \emph{canonical} one.

By a \emph{time-homogeneous Markov process} with transition semigroup $\{P_t\}_{t\in\mathbb{R}_+}$ and initial distribution $\mu\in\mathcal{M}_{prob}(E)$ we mean a family of $E$-valued random variables $\Psi:=\{\Psi(t)\}_{t\in\mathbb{R}_+}$ on some probability space $(\Omega,\mathcal{F},\mathbb{P}_{\mu})$ such that, for any $A\in\mathcal{B}(E)$ and $s,t\geq 0$, \vspace{-0.1cm}
\begin{gather}
\pr_{\mu}(\Psi(0)\in A)=\mu(A),\nonumber\\
\label{def:m_proc}\pr_{\mu}(\Psi(s+t)\in A\,|\,\mathcal{F}(s))=\pr_{\mu}(\Psi(s+t)\in A\,|\,\Psi(s))=P_t(\Psi(s),A),
\end{gather}
where $\mathcal{F}(s)$ is the $\sigma$-field generated by $\{\Psi(h):h\leq s\}$. 

It is not hard to check that formula \eqref{def:m_proc} defines a family $\{P_t\}_{t\in\mathbb{R}_+}$ of stochastic kernels on \hbox{$E\times\mathcal{B}(E)$}, which indeed form a semigroup under the composition operation specified by~\eqref{e:comp}, since $P_{s+t}=P_sP_t$ for any $s,t\geq 0$. The latter is obviously equivalent to saying that the corresponding family of Markov operators is a~Markov semigroup, and therefore implies that $\mu(s+t)=\mu(s)P_t$ for any $s,t\geq 0$, where $\mu(t)$ stands for the distribution of $\Psi(t)$ for every $t\geq 0$. Analogously as in the discrete case, the dual operator of $P_t$ can be expressed as 
\begin{equation}
\label{e:dual_pt}
P_t f(x)=\ew_x[f(\Psi(t))]\quad\text{for any}\quad x\in E,\;f\in B_b(E),\;t\geq 0.
\end{equation}

Let us now briefly recall some notions concerning the ergodicity of Markov operators, which will be used throughout the paper.

First of all, a Markov operator $P$ is called \emph{Feller} if its dual operator preserves continuity, i.e. $P(C_b(E))\subset C_b(E)$. Moreover, a Markov semigroup $\{P_t\}_{t\in\mathbb{R}_+}$ is called \emph{Feller} if $P_t$ is a~Feller operator for any $t\geq 0$. 

A measure $\mu_*\in\mathcal{M}(E)$ is said to be \emph{invariant} for a Markov operator $P$ if \hbox{$\mu_* P=\mu_*$}. By analogy, we say that $\tilde{\mu}_*\in\mathcal{M}(E)$ is \emph{invariant} for a Markov semigroup $\{P_t\}_{t\in\mathbb{R}_+}$ whenever $\tilde{\mu}_* P_t=\tilde{\mu}_*$ for every $t\geq 0$. 

We finalize this section with the definitions of two properties that will be verified in the main results of this paper.\pagebreak
\begin{df}\label{def:erg_p}
Let $P$ be a transition operator of an $E$-valued Markov chain~$\Phi$. Given a Borel measurable function $V:E\to [0,\infty)$, we shall say that $P$ (or the chain $\Phi$) is \hbox{\emph{$V$-exponentially ergodic}} in $d_{FM,\rho}$ if it admits a unique invariant probability measure $\mu_*^{\Phi}$, such that $\mu_*^{\Phi}\in\mathcal{M}_{prob}^V(E)$, and there exists a constant $q\in (0,1)$ such that, for every $\mu\in\mathcal{M}_{prob}^V(E)$ and some \hbox{$C(\mu)<\infty$}, we have
$$d_{FM,\rho}\left(\mu P^n, \mu_*^{\Phi}\right)\leq C(\mu)q^n\quad\text{for any}\quad n\in\n.$$
\end{df}

\begin{df}\label{def:erg_pt}
Let $\{P_t\}_{t\in\mathbb{R}_+}$ be a transition semigroup of an $E$-valued Markov process~$\Psi$. Given a Borel measurable function $V:E\to [0,\infty)$, we shall say that $\{P_t\}_{t\in\mathbb{R}_+}$ (or the process $\Psi$) is \emph{$V$-exponentially ergodic} in~$d_{FM,\rho}$ if it admits a unique invariant probability measure $\mu_*^{\Psi}$, such that \hbox{$\mu_*^{\Psi}\in\mathcal{M}_{prob}^V(E)$}, and there exists a constant $\gamma>0$ such that, for every $\mu\in\mathcal{M}_{prob}^V(E)$ and some \hbox{$\bar{C}(\mu)<\infty$}, we have
$$d_{FM,\rho}\left(\mu P_t, \mu_*^{\Psi}\right)\leq \bar{C}(\mu)e^{-\gamma t}\quad\text{for any}\quad t\geq 0.$$
\end{df}

\section{The model under study} \label{sec:3}
Let $(Y,\rho_Y)$ be a complete separable metric space, and let $I$ be a finite set endowed with the discrete metric $\dd$, i.e. $\dd(i,j)=1$ if $i\neq j$ and $\dd(i,j)=0$ otherwise. In what follows, we shall also refer to the spaces
$$X:=Y\times I \quad \text{and} \quad \bar{X}:=X\times\mathbb{R}_+,$$ 
considered with the product topologies. Additionally, we assume that $X$ is endowed with a~metric $\rho_{X,c}$ of the form
\begin{equation}\label{def:rho_c}
\rho_{X,c}(x_1,x_2)=\rho_Y(y_1,y_2)+c\,\dd(i_1,i_2)\quad\text{for}\quad x_1=(y_1,i_1),\,x_2=(y_2,i_2)\in X,
\end{equation}
where $c$ is a given positive constant, whose value will be relevant in Section \ref{sec:main}. The Fortet-Mourier distance in $\mathcal{M}(X)$, induced by the metric $\rho_{X,c}$, will be simply denoted by $d_{FM,c}$ (rather than $d_{FM,\rho_{X,c}}$). Throughout the paper, we will also refer to the standard bounded metric induced by $\rho_{X,c}$, that is,
\begin{equation}\label{def:rho_c_standard}
\bar{\rho}_{X,c}(x_1,x_2):=\rho_{X,c}(x_1,x_2)\wedge 1\quad\text{for any}\quad x_1,x_2\in X,
\end{equation}
where $\wedge$ stands for the minimum.

Consider a collection $\{S_i:\,i\in I\}$ of jointly continuous semiflows acting from $\mathbb{R}_+\times Y$ to~$Y$. By calling $S_i$ a semiflow we mean, as usual, that
$$S_i(s,S_i(t,y))=S_i(s+t,y)\quad\text{and}\quad S_i(0,y)=y\quad\text{for any}\quad s,t\in\mathbb{R}_+,\;y\in Y.$$

Furthermore, suppose that we are given a right stochastic matrix $\{\pi_{ij}:\;i,j\in I\}$, i.e. 
$$\pi_{ij}\in\mathbb{R}_+\;\;\text{for any}\;\; i,j\in I,\quad\text{and}\quad\sum_{j\in I} \pi_{ij}=1\;\;\text{for every}\;\; i\in I,$$ 
a positive constant $\lambda$, as well as an arbitrary stochastic kernel $\p:Y\times \mathcal{B}(Y)\to [0,1]$.

Let us now define a stochastic kernel $\bar{P}:\mathcal{B}(\bar{X})\times \bar{X}\to [0,1]$ by setting
\begin{equation}
\label{def:P_bar}
\bar{P}((y,i,s), \bar{A})=\sum_{j\in I} \pi_{ij} \int_0^{\infty} \lambda e^{-\lambda h}\int_Y \mathbbm{1}_{\bar{A}}(u,j,h+s)\,\p(S_i(h,y),du)\,dh
\end{equation}
for any $y\in Y$, $i\in I$, $s\in\mathbb{R}_+$ and $\bar{A}\in\mathcal{B}(\bar{X})$. Moreover, let $P:X\times\mathcal{B}(X)\to [0,1]$ denote the kernel given by
\begin{equation}
\label{def:P}
P((y,i),A):=\bar{P}((y,i,0),A\times\mathbb{R}_+)\quad\text{for}\quad y\in Y,\;i\in I,\; A\in\mathcal{B}(X),
\end{equation}
where $\bar{P}$ is given by \eqref{def:P_bar}.

\begin{rem}\label{rem:feller}
Taking into account the continuity of the maps $t \mapsto S_i(t,\cdot)$, $t\geq 0$, $i\in I$, it is easy to see that $P$ is Feller whenever so is the kernel $\p$.
\end{rem}

By $\bar{\Phi}:=\{(Y_n,\xi_n,\tau_n)\}_{n\in\n_0}$ we will denote a time-homogeneous Markov chain with state space $\bar{X}$ and transition law~$\bar{P}$, wherein $Y_n$, $\xi_n$, $\tau_n$ take values in $Y$, $I$, $\mathbb{R}_+$, respectively. More precisely, $\bar{\Phi}$ will be regarded as the canonical Markov chain, constructed on the coordinate space $\Omega:=\bar{X}^{\n_0}$, equipped with the $\sigma$-field $\mathcal{F}:=\mathcal{B}\left(\bar{X}^{\n_0}\right)$ and a suitable family \hbox{$\{\mathbb{P}_{\nu}:\,\nu\in\mathcal{M}_{prob}(\bar{X})\}$} of probability measures on $\mathcal{F}$, where the subscript $\nu$ indicates the initial distribution of $\bar{\Phi}$. For every $\nu\in\mathcal{M}_{prob}(\bar{X})$, we therefore have
\begin{gather*}
\pr_{\nu}(\bar{\Phi}_0\in \bar{A})=\nu(\bar{A})\quad\text{for any}\quad \bar{A}\in\mathcal{B}(\bar{X}),\\
\pr_{\nu}(\bar{\Phi}_{n+1}\in \bar{A}\,|\,\bar{\Phi}_n=(y,i,s))=\bar{P}((y,i,s),\bar{A})\quad\text{for any}\quad (y,i,s)\in\bar{X},\;\bar{A}\in\mathcal{B}(\bar{X}),\;n\in\n_0.
\end{gather*}

Obviously, the sequences $\Phi:=\{(Y_n,\xi_n)\}_{n\in \n_0}$, $\{\xi_n\}_{n\in\n_0}$ and $\{\tau_n\}_{n\in\n_0}$ are Markov chains with respect to their own natural filtrations, and their transition laws satisfy
\begin{gather}
\pr_{\nu}(\Phi_{n+1}\in A\,|\,\Phi_n=(y,i))=P((y,i),A)\quad\text{for}\quad (y,i)\in X,\; A\in\mathcal{B}(X),\nonumber\\
\pr_{\nu}(\xi_{n+1}=j\,|\,\xi_n=i)=\pi_{ij}\quad\text{for}\quad i,j\in I\nonumber\\
\label{e:tau}\pr_{\nu}(\tau_{n+1}\leq t\,|\,\tau_n=s)=\mathbbm{1}_{[s,\infty)}(t)\left(1-e^{-\lambda(t-s)}\right)\quad\text{for}\quad s,t\in\mathbb{R}_+.
\end{gather}
Moreover, note that the increments $\Delta\tau_n:=\tau_{n}-\tau_{n-1}$, $n\in\n$, form a sequence of independent, exponentially distributed random variables with the same rate parameter $\lambda$, and thus $\tau_n\uparrow\infty$, as $n\to\infty$, $\pr_{\nu}$-a.s. (due to the strong law of large numbers).

The main focus of our study will be a PDMP $\Psi:=\{(Y(t),\xi(t))\}_{t\in\mathbb{R}_+}$ with jump times~$\tau_n$, $n\in\n_0$, defined via interpolation of the chain $\Phi$, so that:
\begin{equation}
\label{def:pdmp}
Y(t):=S_{\xi_n}(t-\tau_n,Y_n),\quad \xi(t)=\xi_n\quad\text{whenever}\quad t\in [\tau_n,\tau_{n+1})\;\;\text{for}\;\;n\in\n_0.
\end{equation}
The transition semigroup of this process will be denoted by $\{P_t\}_{t\in\mathbb{R}_+}$. Obviously, the discrete-time model $\Phi$ with transition law $P$, determined by \eqref{def:P}, can be viewed as the Markov chain given by the post-jump locations of $\Psi$, since 
$$\Phi_n=(Y_n,\xi_n)=(Y(\tau_n),\xi(\tau_n))=\Psi(\tau_n) \quad\text{for every} \quad n\in\n_0.$$ 
Looking at the shape of the kernel $\bar{P}$, one can say (somewhat informally) that the probability of visiting a given set $B$ right after the $(n+1)$th jump, given $S_{\xi_n}(\Delta\tau_{n+1},Y_n)=y$, is equal to $\p(y,B)$.

\begin{rem}\label{rem:spec}
As has been already mentioned in the introduction, the above-described model $(\Phi,\Psi)$ is a generalization of that considered in~\cite{b:czapla_erg} (apart from the probabilities $\pi_{ij}$, which are constant here); cf. also \cite{b:czapla_erg2}. More specifically, in~\cite{b:czapla_erg}, the kernel $\p$ is a transition law of some randomly perturbed iterated function system, i.e. it has the form:
$$
\p(y,B)=\int_{\supp \nu}\int_{\Theta}\mathbbm{1}_B(w_{\theta}(y)+v)\,p_{\theta}(y)\,\vartheta(d\theta)\nu(dv)\quad\text{for}\quad B\in\mathcal{B}(Y).
$$
In that case, $Y$ is a closed subset of a Banach space $H$, $\nu\in\mathcal{M}_{prob}(H)$ is a probability measure with bounded support, $\Theta$ stands for an arbitrary topological space, endowed with a Borel measure $\nu$, $\{w_{\theta}:\,\theta\in\Theta\}$ is a given family of continuous transformations from $Y$ to itself, such that $w_{\theta}(Y)+v\subset Y$ for any $v\in \supp\nu$, and $\{p_{\theta}:\,\theta\in\Theta\}$ is an associated set of place-dependent probabilities, i.e. functions from $Y$ to $[0,1]$ such that $\int_{\Theta} p_{\theta}(y)\,\vartheta(d\theta)=1$ for every $y\in Y$. We shall come back to this particular case in Section \ref{sec:ifs}.
\end{rem}

\section{Correspondence between invariant distributions of $\Psi$ and $\Phi$}\label{sec:4}
In the first part of the study, we shall establish a one-to-one correspondence between invariant probability measures for the transition operator $P$ of the chain~$\Phi$, induced by \eqref{def:P},  and those for the transition semigroup $\{P_t\}_{t\in \mathbb{R}_+}$ of the process $\Psi$, specified by \eqref{def:pdmp}. For this aim, let us consider the Markov operators $G,W:\mathcal{M}(X)\to\mathcal{M}(X)$ generated by the stochastic kernels given by
\begin{gather}
\label{def:g}
G((y,i),A)=\int_0^{\infty}\lambda e^{-\lambda t} \mathbbm{1}_A(S_i(t,y),i)\,dt,\\
\label{def:w}
W((y,i),A):=\sum_{j\in I} \pi_{ij}\int_Y \mathbbm{1}_A(u,j)J(y,du)
\end{gather}
for any $(y,i)\in X$ and $A\in\mathcal{B}(X)$. It is easy to check that $GW=P$.

Having defined such operators, we can state the following result:

\begin{thm}\label{tw:rel}\label{thm:cor}
Let $P$ and $\{P_t\}_{t\in\mathbb{R}_+}$ denote the transition operator of the chain $\Phi$ and the transition semigroup of the PDMP $\Psi$, respectively. Further, let $J$ be the kernel appearing in \eqref{def:P_bar}, and suppose that 
\begin{equation}
\label{e:j_cont}
Y\times\mathbb{R}_+\ni(y,t)\mapsto \p g(\cdot,t)(y) \quad \text{is jointly continuous for any} \quad g\in C_b(Y\times\mathbb{R}_+). \vspace{-0.3cm}
\end{equation}
Then
\begin{itemize}
\item[(i)]\label{cnd:1r} if $P$ admits an invariant probability measure $\mu_*^{\Phi}$, then $\mu_*^{\Psi}:=\mu_*^{\Phi} G$ is an invariant  measure for $\{P_t\}_{t\in\mathbb{R}_+}$, and $\mu_*^{\Psi} W = \mu_*^{\Phi}$;
\item[(ii)]\label{cnd:2r} if $\{P_t\}_{t\in\mathbb{R}_+}$ has an invariant probability measure $\mu_*^{\Psi}$, then \hbox{$\mu_*^{\Phi}:=\mu_*^{\Psi} W$} is an invariant measure for $P$, and $\mu_*^{\Phi} G=\mu_*^{\Psi}$.
 \end{itemize}
\end{thm}

\begin{rem}
Obviously, if \eqref{e:j_cont} holds, then, in particular, $J$ is Feller.
\end{rem}

Theorem \ref{thm:cor} can be proved exactly in the same way as  \hbox{\cite[Theorem 4.4]{b:czapla_erg}}, which refers to the case where the kernel $\p$, appearing in \eqref{def:P_bar}, is defined explicitly (as mentioned in Remark~\ref{rem:spec}). In order to adapt this proof to our setting (with an arbitrary stochastic kernel~$\p$), one only needs to establish the properties collected in the lemma below.
\begin{lem}\label{lem:Pt_prop}
The following statements hold for the transition semigroup $\{P_t\}_{t\in\mathbb{R}_+}$ of the process~$\Psi$:
\begin{enumerate}[label=(\roman*)]
\item\label{lem:i}  If $\p$ is Feller, then $\{P_t\}_{t\in\mathbb{R}_+}$ is Feller.
\item\label{lem:ii} For any $f\in B_b(X)$, there exists a bounded Borel measurable function $u_f:X\times \mathbb{R}_+\to\mathbb{R}$ such that $\lim_{t\to 0} \norma{u_f(\cdot,t)}_{\infty}/t=0$, and 
$$P_t f(y,i)=e^{-\lambda t} f(S_i(t,y),i)+\lambda e^{-\lambda t}\int_0^t \psi_f((y,i),s,t)\,ds+u_f((y,i),t),$$
for any $(y,i)\in X$ and $t>0$, where 
$$\psi_f((y,i),s,t):=\sum_{j\in I}\pi_{ij}\int_Y f(S_j(t-s,u),j)\p(S_i(s,y),du)\quad\text{for}\quad s\in [0,t],\;t>0.$$
\item\label{lem:iii} $\{P_t\}_{t\in\mathbb{R}_+}$ is stochastically continuous, i.e. 
$$\lim_{t\to 0} P_t f(y,i)=f(y,i)\quad\text{for any}\quad (y,i)\in X,\; f\in C_b(X).$$
\item\label{lem:iv}  If \eqref{e:j_cont} holds then, for any function \hbox{$s:\mathbb{R}_+\to\mathbb{R}_+$} satisfying $0\leq s(t)\leq t$ for all $t\geq 0$, the map $t\mapsto \psi_f((y,i),s(t),t)$ is continuous at $t=0$ for every $f\in C_b(X)$ and any $(y,i)\in X$. Moreover, $\psi_f(\cdot,0,0)=Wf$, where $W$ is the operator induced by \eqref{def:w}.
\end{enumerate}
\end{lem}
\begin{proof}
Throughout the proof, we will write $\bar{x}:=(x,0)$ for any given $x\in X$. 
Moreover, for every~$i\in I$, we put $f(S_i(h,\cdot)):=0$ if $h<0$.

Let $t\in\mathbb{R}_+$ and $f\in B_b(X)$. Then, according to \eqref{e:dual_pt}, for every $x\in X$, we can~write\vspace{-0.3cm}
\begin{align}
\begin{split}
\label{e:pt_ew}
P_t f(x)&=\ew_{\bar{x}}f(Y(t),\xi(t))=\sum_{n=0}^{\infty} \ew_{\bar{x}}\left[\mathbbm{1}_{[\tau_n,\tau_{n+1})}(t)f(S_{\xi_n}(t-\tau_n,Y_n),\xi_n) \right]\\
&=\sum_{n=0}^{\infty} \ew_{\bar{x}}\left[\mathbbm{1}_{[0,t]}(\tau_n)f(S_{\xi_n}(t-\tau_n,Y_n),\xi_n)\cdot\mathbbm{1}_{(t,\infty)}(\tau_{n+1}) \right].
\end{split}
\end{align}
Taking into account \eqref{e:canonical}, it is clear that, for any $g,h\in B_b(\bar{X})$ and $n\in\n_0$,
\begin{align}
\begin{split}
\label{e:dod}
\ew_{\bar{x}}[g(\bar{\Phi}_n)h(\bar{\Phi}_{n+1})]&=\int_{\bar{X}}\int_{\bar{X}} g(w)h(z)\,\bar{P}(w,dz)\bar{P}^n(\bar{x},dw)\\
&=\int_{\bar{X}}g(w)\bar{P}h(w)\bar{P}^n(\bar{x},dw)=\bar{P}^n(g\bar{P}h)(\bar{x}).
\end{split}
\end{align}
Hence, defining
\begin{equation}
\label{def:gh}
g_t(u,j,s):=\mathbbm{1}_{[0,t]}(s)f(S_j(t-s,u),j)\quad\text{and}\quad
h_t(u,j,s):=\mathbbm{1}_{(t,\infty)}(s)
\end{equation}
for any $u\in Y$, $j\in I$ and $s\in\mathbb{R}_+$, we see that
\begin{align}
\begin{split}
\label{e:decomp}
P_t f(x)&=\sum_{n=0}^{\infty}\ew_{\bar{x}}[g_t(Y_n,\xi_n,\tau_n)h_t(Y_{n+1},\xi_{n+1},\tau_{n+1})]\\&=\sum_{n=0}^{\infty}\bar{P}^n(g_t\bar{P}h_t)(\bar{x})\quad\text{for every}\quad x\in X.
\end{split}
\end{align}

\textbf{\ref{lem:i}}: Suppose that the kernel $\p$ is Feller, and that $f\in C_b(X)$. To prove that $P_tf$ is continuous, we first observe that, for any function $\varphi\in B_b(\bar{X})$ such that $X\ni x\mapsto\varphi(x,s)$ is continuous for every $s\geq 0$ (which is the case for $g_t$ and $h_t$), the map \hbox{$X\ni x\mapsto \bar{P}\varphi(x,s)$} is continuous for any $s\geq 0$ as well, since
$$\bar{P}\varphi(x,s)=\sum_{j\in I}\pi_{ij}\int_0^{\infty}\lambda e^{-\lambda h}\, \p \varphi(\cdot,j,h+s)(S_i(h,y))\,dh\;\;\text{for any}\;\;x=(y,i)\in X,\; s\geq 0.$$
This implies that all the maps $X\ni x\mapsto \bar{P}^n(g_t\bar{P}h_t)(\bar{x})$, $t\geq 0$, $n\in \n_0$, are continuous. Further, considering the Poisson process $(N(s))_{s\in\mathbb{R}_+}$ of the form
\begin{equation}
\label{def:poisson}
N(s):=\max\{n\in\n_0:\,\tau_n\leq s\},\quad s\geq 0,
\end{equation}
and bearing in mind \eqref{e:dod}, we see that
\begin{align}
\begin{split}
\label{e:com_bound}
|\bar{P}^n(g_t\bar{P}h_t)(\bar{x})|&=\left|\ew_{\bar{x}}[\mathbbm{1}_{\{N(t)=n\}}f(S_{\xi_n}(t-\tau_n,Y_n),\xi_n)) ]\right|\leq \norma{f}_{\infty}\pr_{\bar{x}}(N(t)=n)\\
&=\norma{f}_{\infty}e^{-\lambda t}\frac{(\lambda t)^n}{n!}\quad \text{for any}\quad \bar{x}\in\bar{X}, \;n\in\n_0.
\end{split}
\end{align}
Using \eqref{e:decomp}, \eqref{e:com_bound} and the discrete analogue of the Lebesgue dominated convergence theorem, we can therefore conclude that $P_tf$ is indeed continuous.

\textbf{\ref{lem:ii}}: Let us define
$$u_f(x,t):=\sum_{n=2}^{\infty}\bar{P}^n(g_t\bar{P}h_t)(\bar{x})\quad\text{for}\quad x\in X.$$ 
From \eqref{e:decomp} it now follows that
\begin{equation}
\label{e:decomp2}
P_t f(x)=g_t\bar{P}h_t(\bar{x})+\bar{P}(g_t\bar{P}h_t)(\bar{x})+u_f(x,t)\quad\text{for any}\quad x\in X.
\end{equation}
Referring to \eqref{e:com_bound}, we get
\begin{align}
\begin{split}
\label{e:est_rest}
\frac{|u_f(x,t)|}{t}&\leq \norma{f}_{\infty}\frac{1}{t}e^{-\lambda t}\sum_{n=2}^{\infty} \frac{(\lambda t)^n}{n!}=\norma{f}_{\infty}\frac{1}{t}e^{-\lambda t}(e^{\lambda t}-1-\lambda t)\\
&=\norma{f}_{\infty}\left(\frac{1-e^{-\lambda t}}{t}-\lambda e^{-\lambda t} \right)
\end{split}
\end{align}
for any $x\in X$, which obviously gives $\lim_{t\to 0} \norma{u_f(\cdot,t)}_{\infty}/t=0$. Further, having in mind \eqref{def:P_bar} and \eqref{def:gh}, we obtain
\begin{align*}
g_t\bar{P}h_t(y,i,s)&=\mathbbm{1}_{[0,t]}(s)f(S_i(t-s,y),i)\int_0^{\infty}\lambda e^{-\lambda h}\mathbbm{1}_{(t,\infty)}(h+s)\,dh
\\
&=\mathbbm{1}_{[0,t]}(s)f(S_i(t-s,y),i) e^{-\lambda(t-s)}\quad\text{for any}\quad y\in Y,\; i\in I,\; s\geq 0,
\end{align*}
and, in particular (for $s=0$),
\begin{equation}
\label{e:gph}
g_t\bar{P}h_t(\bar{x})=g_t\bar{P}h_t(y,i,0)=e^{-\lambda t}f(S_i(t,y),i)\quad\text{for every}\quad x=(y,i)\in X.
\end{equation}
Furthermore, appealing to \eqref{def:P_bar} and \eqref{e:gph}, we can also conclude that
\begin{align}
\begin{split}
\label{e:ngph}
\bar{P}(g_t\bar{P}h_t)(\bar{x})&=\sum_{j\in I} \pi_{ij} \int_0^{\infty} \lambda e^{-\lambda h} \int_Y g_t(u,j,h)\bar{P}h_t(u,j,h)\,\p(S_i(h,y),du)\,dh\\
&=\sum_{j\in I} \pi_{ij} \int_0^{\infty} \lambda e^{-\lambda h}\int_Y\mathbbm{1}_{[0,t]}(h)f(S_j(t-h,u),j) e^{-\lambda(t-h)}\,\p(S_i(h,y),du)\,dh\\
&=\lambda e^{-\lambda t}\int_0^t \sum_{j\in I}\pi_{ij}\int_Y f(S_j(t-h,u),j)\,\p(S_i(h,y),du)\,dh\\
&=\lambda e^{-\lambda t}\int_0^t \psi_f((y,i),h,t)\,dh \quad\text{for all}\quad x=(y,i)\in X.
\end{split}
\end{align}
Finally, assertion \ref{lem:ii} follows from \eqref{e:decomp2}-\eqref{e:ngph}.

\textbf{\ref{lem:iii}}: Condition \ref{lem:iii} follows immediately from \ref{lem:ii} and the boundedness of $\psi_f$.

\textbf{\ref{lem:iv}}: For the proof of (iv), fix $f\in C_b(X)$, $(y,i)\in X$, and define
$$g(u,t):=\sum_{j\in I} \pi_{ij} f(S_j(t-s(t),u),j) \quad\text{for}\quad (u,t)\in Y\times\mathbb{R}_+.$$
Since $S_j$ are jointly continuous, so is $g$, and thus $g\in C_b(Y\times\mathbb{R}_+)$. Moreover, we have $$Jg(\cdot,t)(S_i(s(t),y))=\psi_f((y,i),s(t),t)\quad\text{for any}\quad t\geq 0.$$ Hence, $s\mapsto \psi_f((y,i),s(t),t)$ is continuous at $0$ whenever \eqref{e:j_cont} holds and $s(t)\to 0$, as $t\to 0$. The identity $\psi_f(\cdot,0,0)=Wf$ is just a consequence of the definition of $W$.
\end{proof}
%

\section{A coupling argument (involving $\Phi$) for establishing the exponential ergodicity in $d_{FM,c}$ of both $\Phi$ and $\Psi$}\label{sec:coup}
The main goal of this section is to prove that the existence of an appropriate coupling between two copies of the chain $\Phi$ (starting from different initial distributions), which brings them closer to each other on average at a geometric rate, combined with a Foster-Lyapunov type condition related to $P$ and some reasonable assumption on the semiflows $S_i$, implies that both $\Phi$ and $\Psi$ are exponentially ergodic in the Fortet--Mourier metric.

More specifically, we shall consider a coupling between two copies of the chain $\Phi$, enhanced with a copy of $\{\tau_n\}_{n\in\n_0}$, that is, a time-homogeneous Markov chain $\widehat{\Phi}:=\{(\Phi_n^{(1)},\Phi_n^{(2)}, \widetilde{\tau}_n)\}_{n\in\n_0}$ evolving on \hbox{$Z:=X^2\times\mathbb{R}_+$}, whose transition law $\widehat{P}:Z\times\mathcal{B}(Z)\to [0,1]$ satisfies
\begin{equation}\label{def:coup}
\begin{gathered}
\widehat{P}((x_1,x_2,s),A\times X\times\mathbb{R}_+)=P(x_1,A),\quad \widehat{P}((x_1,x_2,s),X\times A\times\mathbb{R}_+)=P(x_2,A),\\
\widehat{P}((x_1,x_2,s),X^2\times T)=\int_0^{\infty} \lambda e^{-\lambda t}\mathbbm{1}_T(t+s)\,dt=:E_{\lambda}(s,T)
\end{gathered}
\end{equation}
for any $x_1,x_2\in X$, $s\geq 0$, $A\in\mathcal{B}(X)$ and $T\in\mathcal{B}(\mathbb{R}_+)$.

Such an augmented coupling $\widehat{\Phi}$ will be regarded as a canonical Markov chain, defined on the coordinate space $(\widehat{\Omega},\widehat{\mathcal{F}})$, with $\widehat{\Omega}:=Z^{\mathbb{N}_0}$ and $\widehat{\mathcal{F}}:=\mathcal{B}\left(Z^{\mathbb{N}_0}\right)$, equipped with an appropriate collection  $\{\widehat{\pr}_{(\mu_1,\mu_2)}:\,\mu_1,\mu_2\in\mathcal{M}_{prob}(X)\}$ of probability measures on $\widehat{\mathcal{F}}$ such that
\begin{gather*}
\widehat{\pr}_{(\mu_1,\mu_2)}(\widehat{\Phi}_0\in D)=(\mu_1\otimes \mu_2\otimes\delta_0)(D)\quad\text{for any}\quad D\in\mathcal{B}(Z),\\
\widehat{\pr}_{(\mu_1,\mu_2)}(\widehat{\Phi}_{n+1}\in D\,|\,\widehat{\Phi}_n=z)=\widehat{P}(z,D)\quad\text{for any}\quad z\in Z,\;D\in\mathcal{B}(Z),\;n\in\n_0,
\end{gather*}
where $\delta_0$ stands for the Dirac measure at $0$ on $\mathcal{B}(\mathbb{R}_+)$. The expectation operator corresponding to $\widehat{\pr}_{(\mu_1,\mu_2)}$ will be denoted by $\widehat{\ew}_{(\mu_1,\mu_2)}$. In the case where $\mu_1=\delta_{x_1}$ and $\mu_2=\delta_{x_2}$ with some $x_1,x_2\in X$, we will write $(x_1,x_2)$ instead of $(\delta_{x_1},\delta_{x_2})$ in the subscripts.

We begin the analysis by establishing a general connection between the exponential ergodicity of $P$ and the existence of an appropriate coupling of $\Phi$, based on a Foster-Lyapunov type condition. It is worth noting that the result, in fact, does not depend on the shape of the transition law $P$.
\begin{lem}\label{lem:a}
Suppose that $P$ is Feller. Furthermore, assume that the transition law $\widehat{P}$ of the chain $\widehat{\Phi}$, satisfying \eqref{def:coup}, can be constructed so that
\begin{equation}
\label{e:coup_prop_v}
\widehat{\ew}_{(x_1,x_2)}\left[\bar{\rho}_{X,c}\left(\Phi_n^{(1)},\Phi_n^{(2)}\right)\right]\leq C_0(V(x_1)+V(x_2)+1)q^n\;\;\text{for all}\;\; n\in\n,\; x_1,x_2\in X,
\end{equation}
where is $\bar{\rho}_{X,c}$ is given by \eqref{def:rho_c_standard}, $q\in (0,1)$, $C_0<\infty$, and $V:X\to [0,\infty)$ is a continuous function such that, for some $a\in (0,1)$ and some $b\in (0,\infty)$, we have
\begin{equation}
\label{e:lyapunov}
PV(x)\leq aV(x)+b\quad\text{for all}\quad x\in X.
\end{equation}
Then there exists a unique invariant distribution $\mu_*^{\Phi}$ for $P$ such that $\mu_*^{\Phi}\in\mathcal{M}_{prob}^V(X)$ and 
\begin{equation}
\label{e:erg_disc}
d_{FM,c}(\mu P^n,\mu_*^{\Phi})\leq C_0(\<V,\mu\>+\<V,\mu_*^{\Phi}\>+1)q^n\;\;\text{for any}\;\;n\in\n,\; \mu\in\mathcal{M}_{prob}(X).
\end{equation}
\end{lem}
\begin{proof}
First of all, note that, for any $\mu_1,\mu_2\in\mathcal{M}_{prob}(X)$ and $n\in\n$,
\begin{equation}
\label{e:mix}
d_{FM,c}(\mu_1 P^n, \mu_2 P^n)\leq C_0(\<V,\mu_1\>+\<V,\mu_2\>+1)q^n.
\end{equation}
To see this, it suffices to observe that, for every $f\in \Lip_{b,1}(X)$,
\begin{align*}
|\<f,\mu_1 P^n-\mu_2 P^n\>|&=\left|\widehat{\ew}_{(\mu_1,\mu_2)}\left[f(\Phi_n^{(1)})-f(\Phi_n^{(2)})\right]\right|\leq \widehat{\ew}_{(\mu_1,\mu_2)}\left|f(\Phi_n^{(1)})-f(\Phi_n^{(2)})\right|\\
&\leq  \widehat{\ew}_{(\mu_1,\mu_2)}\left[\bar{\rho}_{X,c}\left(\Phi_n^{(1)},\Phi_n^{(2)}\right)\right]\\
&=\int_{X^2} \widehat{\ew}_{(x_1,x_2)}\left[\bar{\rho}_{X,c}\left(\Phi_n^{(1)},\Phi_n^{(2)}\right)\right](\mu_1\otimes\mu_2)(dx_1,dx_2)\\
&\leq C_0 \left(\int_X\int_X(V(x_1)+V(x_2)+1)\,\mu_1(dx_1)\mu_2(dx_2)\right)q^n\\
&=C_0(\<V,\mu_1\>+\<V,\mu_2\>+1)q^n,
\end{align*}
where the first equality follows from \eqref{e:dual}, and the second one is due to \eqref{e:p_m}.

The next step is to prove that $P$ admits an invariant probability measure. For this purpose, let us fix arbitrarily $x_0\in X$ and notice that $\{\delta_{x_0} P^n\}_{n\in\n}$ is a Cauchy sequence in the metric space $(\mathcal{M}_{prob}(X), d_{FM,c})$. Indeed, applying~\eqref{e:mix} with $\mu_1=\delta_{x_0}$ and $\mu_2=\delta_{x_0} P^k$ (for each $k\in\n_0$), together with \eqref{e:lyapunov}, we can conclude that
\begin{align*}
d_{FM,c}(\delta_{x_0} P^n, \delta_{x_0} P^{k+n})&\leq C_0(V(x_0)+P^kV(x_0)+1)q^n\\
& \leq C_0\left(V(x_0)+a^k V(x_0)+\frac{b}{1-a}+1\right)q^n\\
&\leq C_0\left(2V(x_0)+\frac{b}{1-a}+1\right)q^n \quad\text{for any}\quad n\in\n,\;k\in\n_0.
\end{align*}
Consequently, since $(\mathcal{M}_{prob}(X), d_{FM,c})$ is complete, $\{\delta_{x_0} P^n\}_{n\in\n}$ is weakly convergent to some \hbox{$\mu_*\in\mathcal{M}_{prob}(X)$}. From the Feller property it follows that $\mu_*$ is invariant for~$P$, since
$$\<f,\mu_* P\> =\<Pf,\mu_*\> = \lim_{n\to\infty}\<Pf,\delta_{x_0} P^n\>=\lim_{n\to\infty}\<f,\delta_{x_0} P^{n+1}\>=\<f,\mu_*\>\quad\text{for any}\quad f\in C_b(X).$$ 
Obviously, \eqref{e:mix}, together with the invariance of $\mu_*$, ensure that \eqref{e:erg_disc} holds with $\mu_*^{\Phi}:=\mu_*$.

In order to show that $\mu_*\in\mathcal{M}_{prob}^V(X)$, consider $V_k(x):=\min (V(x),k)$ for any $x\in X$ and $k\in\n$. Then $\{V_k\}_{k\in\n}$ is a non-decreasing sequence of functions of $C_b(X)$. From \eqref{e:lyapunov} it follows that
$$P^n V_k \leq P^n V\leq a^n V+\frac{b}{1-a}\quad\text{for all}\quad k,n\in\n,\vspace{-0.3cm}$$
whence
$$\<V_k,\mu_*\>=\lim_{n\to\infty} \<V_k,\delta_{x_0} P^n\> =\lim_{n\to\infty}P^n V_k(x_0)\leq \frac{b}{1-a}\quad\text{for every}\quad k\in\n.$$
By using the Lebesgue monotone convergence theorem, we therefore obtain
\begin{equation}
\label{e:moment}
\<V,\mu_*\>=\lim_{k\to\infty} \<V_k,\mu_*\>\leq \frac{b}{1-a}<\infty.
\end{equation}

What is left is to show that there are no other invariant measures for $P$. To this end, it is enough to know
\begin{equation}
\label{e:asymp}
\lim_{n\to\infty} d_{FM,c}(\mu P^n, \mu_*)=0\quad\text{for any}\quad\mu\in\mathcal{M}_{prob}(X),
\end{equation}
but this can be easily derived from \eqref{e:erg_disc} and \eqref{e:moment}. More precisely, these conditions guarantee that \eqref{e:asymp} holds for any $\mu\in\mathcal{M}_{prob}^V(X)$, and so, in particular, we have 
$$P^n f(x)=\<f,\delta_x P^n\>\to \<f,\mu_*\>,\;\;\text{as}\;\;n\to \infty,\;\; \text{for any}\;\;x\in X\;\; \text{and} \;\;f\in C_b(X).$$ 
Now, letting $\mu$ be an arbitrary probability measure and applying the Lebesgue dominated convergence theorem, we obtain
$$\<f, \mu P^n\>=\<P^n f,\mu\>\to \<f,\mu_*\>,\;\;\text{as}\;\;n\to\infty,\;\;\text{for any}\;\; f\in C_b(X),$$
which obviously gives \eqref{e:asymp} and completes the proof.
\end{proof}

Given an augmented coupling  $\widehat{\Phi}:=\{(\Phi_n^{(1)},\Phi_n^{(2)}, \widetilde{\tau}_n)\}_{n\in\n_0}$ between any two copies of $\Phi$, and writing
$$\Phi_n^{(i)}=(Y_n^{(i)},\xi_n^{(i)})\quad\text{for}\quad n\in\n_0,\;i\in\{1,2\},$$
to indicate the coordinates $Y_n^{(i)}$ and $\xi_n^{(i)}$  with values in $Y$ and $I$, respectively, we can consider the two corresponding copies $\Psi^{(1)}$ and $\Psi^{(2)}$ of the process $\Psi$, defined as follows:
$$
\Psi^{(i)}(t):=(Y^{(i)}(t),\xi^{(i)}(t))\quad\text{for}\quad t\geq 0,\;i\in\{1,2\},
\vspace{-0.4cm}
$$
where
$$
Y^{(i)}(t):=S_{\xi_n^{(i)}}\left(t-\widetilde{\tau}_n,Y_n^{(i)}\right),\;\;\xi^{(i)}(t)=\xi_n^{(i)}\quad\text{whenever}\quad t\in\left[\widetilde{\tau}_n,\widetilde{\tau}_{n+1}\right),\; n\in\n_0,\;i\in\{1,2\}.
$$

Our aim now is to show that any two copies of the process $\Psi$, defined as above on the path space of $\widehat{\Phi}$, get closer to each other on average at an exponential rate, whenever $\widehat{\Phi}$ satisfies \eqref{e:coup_prop_v} and a Lipschitz type condition is imposed on the semiflows $S_i$. This will be a crucial step in deriving the exponential ergodicity of the semigroup $\{P_t\}_{t\in\mathbb{R}_+}$. The proof of this result, given below, is inspired by ideas developed in \cite{b:cloez_hairer}.

\begin{lem}\label{lem:b}
Let $\mu_1,\mu_2\in\mathcal{M}_{prob}(X)$, and suppose that the transition law $\widehat{P}$ of the chain $\widehat{\Phi}$, satisfying \eqref{def:coup}, can be constructed so that
\begin{equation}
\label{e:coup_prop}
\widehat{\ew}_{(\mu_1,\mu_2)}\left[\bar{\rho}_{X,c}\left(\Phi_n^{(1)},\Phi_n^{(2)}\right)\right]\leq C(\mu_1,\mu_2)q^n\quad\text{for all}\quad n\in\mathbb{N},
\end{equation}
with certain constants \hbox{$q\in (0,1)$} and $C(\mu_1,\mu_2)<\infty$. Furthermore, assume that there exist $L<\infty$ and $\alpha\in(-\infty,\lambda)$ such that
\begin{equation}
\label{e:c}
\rho_Y(S_i(t,u), S_i(t,v))\leq Le^{\alpha t}\rho_Y(u,v)\;\;\;\mbox{for any}\;\;\;u,v\in Y,\;i\in I,\;t\geq 0.
\end{equation}
Then there exist constants $\gamma>0$ and $\bar{C}(\mu_1,\mu_2)<\infty$ for which
\begin{equation}
\label{erg_pt}
\widehat{\ew}_{(\mu_1,\mu_2)}\left[\bar{\rho}_{X,c}\left(\Psi^{(1)}(t),\Psi^{(2)}(t)\right)\right] \leq \bar{C}(\mu_1,\mu_2)\,e^{-\gamma t}\quad\text{for all}\quad t\geq 0.
\end{equation}
Moreover, if $q$ does not depend on the measures $\mu_1,\mu_2$, then $\gamma$ can be chosen so that it does not depend on them too.
\end{lem}
\begin{proof}
Let $\kappa$ be the coupling time for $\left\{\left (\xi_n^{(1)},\xi_n^{(2)}\right)\right\}_{n\in\n_0}$, that is,
$$\kappa:=\inf\{n\in\n_0:\,\xi_n^{(1)}=\xi_n^{(2)}\}.$$
From \eqref{e:coup_prop} it then follows that, for every $n\in\n$,
\begin{align}
\begin{split}
\label{mp:1}
\widehat{\pr}_{(\mu_1,\mu_2)}(\kappa>n)
&\leq \widehat{\pr}_{(\mu_1,\mu_2)}\left(\xi_n^{(1)}\neq \xi_n^{(2)}\right)
=\widehat{\ew}_{(\mu_1,\mu_2)}\left[\mathbbm{1}_{\left\{\xi_n^{(1)}\neq\xi_n^{(2)}\right\}}\right]\\
&=\widehat{\ew}_{(\mu_1,\mu_2)}\left[\dd(\xi_n^{(1)},\xi_n^{(2)})\right]
\leq \widehat{\ew}_{(\mu_1,\mu_2)}\left[\bar{\rho}_{X,c}(\Phi_n^{(1)},\Phi_n^{(2)}) \right]\\
&\leq C(\mu_1,\mu_2)q^n,
\end{split}
\end{align}
which, in particular, shows that $\pr_{(\mu_1,\mu_2)}(\kappa<\infty)=1$. 

In what follows, the processes $\{\xi_n^{(2)}\}_{n\in\n_0}$ and $\{\Psi^{(2)}(t)\}_{t\in\mathbb{R}_+}$ will be identified with their copies $\{\wxi_n^{(2)}\}_{n\in\n_0}$ and $\{\wPsi^{(2)}(t)\}_{t\in\mathbb{R}_+}$, respectively, defined as follows:
\begin{gather*}
\wxi_n^{(2)}:=
\begin{cases}
\xi_n^{(2)} &\mbox{if}\quad n\leq\kappa,\\
\xi_n^{(1)} &\mbox{if}\quad n>\kappa,
\end{cases}\\[0.2cm]
\wPsi^{(2)}(t):=\left(S_{\wxi_n^{(2)}}\left(t-\otau_n,Y_n^{(2)}\right),\wxi_n^{(2)}\right),\quad\text{whenever}\quad t\in\left[\otau_n,\otau_{n+1}\right)\quad\text{for}\quad n\in\n_0.
\end{gather*}
By $(\widetilde{N}(s))_{s\in\mathbb{R}_+}$ we will denote the Poisson process given by
$$\widetilde{N}(s):=\max\{n\in\n_0:\,\otau_n\leq s\}\quad\text{for}\quad s\geq 0.$$
Moreover, we will write $\{\mathcal{F}_n\}_{n\in\n_0}$ for the natural filtration of $\widehat{\Phi}=\left\{\left(\Phi_n^{(1)},\Phi_n^{(2)},\otau_n\right)\right\}_{n\in\n_0}$.

Let $n\in\n_0$ and $t\geq 0$ be arbitrary. Keeping in mind that $\bar{\rho}_{X,c}\leq 1$, we can write
\begin{align}
\label{mp:2}
&\widehat{\ew}_{(\mu_1,\mu_2)}\left[\bar{\rho}_{X,c}\left(\Psi^{(1)}(t),\Psi^{(2)}(t)\right)\mathbbm{1}_{\{\widetilde{N}(t)=n\}}\,|\,\mathcal{F}_n  \right]\nonumber\\
&\leq \widehat{\ew}_{(\mu_1,\mu_2)}\left[\bar{\rho}_{X,c}\left(\Psi^{(1)}(t),\Psi^{(2)}(t)\right)^{1/2}\mathbbm{1}_{\{\widetilde{N}(t)=n\}}\,|\,\mathcal{F}_n  \right]\\
&\leq  \widehat{\ew}_{(\mu_1,\mu_2)}\left[\bar{\rho}_{X,c}\left(\Psi^{(1)}(t),\Psi^{(2)}(t)\right)^{1/2}\mathbbm{1}_{\{\widetilde{N}(t)=n\}}\mathbbm{1}_{\{\kappa\leq n\}}\,|\,\mathcal{F}_n  \right]+\widehat{\ew}_{(\mu_1,\mu_2)}\left[\mathbbm{1}_{\{\widetilde{N}(t)=n\}}\mathbbm{1}_{\{\kappa>n\}}\,|\,\mathcal{F}_n  \right].\nonumber
\end{align}

Defining
$$\bar{L}=\max\{L,1\}\quad\text{and}\quad \bar{\alpha}=\max\{\alpha,0\}.$$
we see that \eqref{e:c} gives
\begin{equation}
\label{e:cc}
\rho_Y(S_i(t,u), S_i(t,v))\wedge 1\leq \bar{L}e^{\bar{\alpha} t}\left[\rho_Y(u,v)\wedge 1\right]\;\;\;\mbox{for any}\;\;\;u,v\in Y,\;i\in I,\;t\geq 0.
\end{equation}

Taking into account that $\{\widetilde{N}(t)=n\}=\{\otau_n\leq t<\otau_{n+1}\}$, and that $\xi_n^{(1)}=\xi_n^{(2)}$ whenever $n\geq \kappa$ (due the adopted identification $\xi^{(2)}=\widetilde{\xi}^{(2)}$), we may apply \eqref{e:cc} to estimate  the first term on the right-hand side of \eqref{mp:2} as follows:
\begin{align}
\label{mp:3}
&\widehat{\ew}_{(\mu_1,\mu_2)}\left[\bar{\rho}_{X,c}\left(\Psi^{(1)}(t),\Psi^{(2)}(t)\right)^{1/2}\mathbbm{1}_{\{\widetilde{N}(t)=n\}}\mathbbm{1}_{\{\kappa\leq n\}}\,|\,\mathcal{F}_n  \right]\nonumber\\
&=\mathbbm{1}_{\{\otau_n\leq t\}} \mathbbm{1}_{\{\kappa\leq n\}}\left[\rho_Y\left(S_{\xi_n^{(1)}}\left(t-\otau_n,Y_n^{(1)}\right),S_{\xi_n^{(2)}}\left(t-\otau_n,Y_n^{(2)}\right)\right)\wedge 1\right]^{1/2}\widehat{\ew}_{(\mu_1,\mu_2)}\left[\mathbbm{1}_{\{\otau_{n+1}>t\}}\,|\,\mathcal{F}_n  \right]\nonumber\\
&\leq \mathbbm{1}_{\{\otau_n\leq t\}} \mathbbm{1}_{\{\kappa\leq n\}} \bar{L}^{1/2}\,e^{\bar{\alpha} (t-\otau_n)/2}\left[\rho_Y\left(Y_n^{(1)},Y_n^{(2)}\right)\wedge 1\right]^{1/2}\widehat{\pr}_{(\mu_1,\mu_2)}(\otau_{n+1}>t\,|\,\mathcal{F}_n)\nonumber\\
&\leq \mathbbm{1}_{\{\otau_n\leq t\}} \mathbbm{1}_{\{\kappa\leq n\}} \bar{L}^{1/2}\,e^{\bar{\alpha} (t-\otau_n)/2}\,\bar{\rho}_{X,c}\left(\Phi_n^{(1)},\Phi_n^{(2)}\right)^{1/2}\widehat{\pr}_{(\mu_1,\mu_2)}(\otau_{n+1}>t\,|\,\mathcal{F}_n).
\end{align}
Since, according to \eqref{e:tau},
$$\widehat{\pr}_{(\mu_1,\mu_2)}(\otau_{n+1}>t\,|\,\mathcal{F}_{n})=\widehat{\pr}_{(\mu_1,\mu_2)}(\otau_{n+1}>t\,|\,\tau_{n})=e^{-\lambda(t-\otau_{n})}\quad\text{on}\quad\{\otau_{n}\leq t\},$$
it follows that
\begin{align}
\begin{split}
\label{mp:4}
\widehat{\ew}_{(\mu_1,\mu_2)}&\left[\bar{\rho}_{X,c}\left(\Psi^{(1)}(t),\Psi^{(2)}(t)\right)^{1/2}\mathbbm{1}_{\{\widetilde{N}(t)=n\}}\mathbbm{1}_{\{\kappa\leq n\}}\,|\,\mathcal{F}_n  \right]\\
&\leq \mathbbm{1}_{\{\otau_n\leq t\}} 
\mathbbm{1}_{\{\kappa\leq n\}} \bar{L}^{1/2}\,e^{-(\lambda-\bar{\alpha}/2)t}e^{(\lambda-\bar{\alpha}/2)\otau_n}\bar{\rho}_{X,c}\left(\Phi_n^{(1)},\Phi_n^{(2)}\right)^{1/2}.
\end{split}
\end{align}

Consequently, due to \eqref{mp:2} and \eqref{mp:4}, we obtain
\begin{align}
\begin{split}
\label{mp:5}
\widehat{\ew}_{(\mu_1,\mu_2)}&\left[\bar{\rho}_{X,c}\left(\Psi^{(1)}(t),\Psi^{(2)}(t)\right)\mathbbm{1}_{\{\widetilde{N}(t)=n\}}\,|\,\mathcal{F}_n  \right]\\
&\leq \mathbbm{1}_{\{\otau_n\leq t\}} 
\mathbbm{1}_{\{\kappa\leq n\}} \bar{L}^{1/2}e^{-(\lambda-\bar{\alpha}/2)t}e^{(\lambda-\bar{\alpha}/2)\otau_n}\bar{\rho}_{X,c}\left(\Phi_n^{(1)},\Phi_n^{(2)}\right)^{1/2}\\
&\quad+\widehat{\ew}_{(\mu_1,\mu_2)}\left[\mathbbm{1}_{\{\widetilde{N}(t)=n\}}\mathbbm{1}_{\{\kappa>n\}}\,|\,\mathcal{F}_n  \right].
\end{split}
\end{align}

Taking the expectation of both sides of the last inequality and using the Cauchy--Schwarz inequality gives
\begin{align}
\begin{split}
\label{mp:6}
&\widehat{\ew}_{(\mu_1,\mu_2)}\left[\bar{\rho}_{X,c}\left(\Psi^{(1)}(t),\Psi^{(2)}(t)\right)\mathbbm{1}_{\{\widetilde{N}(t)=n\}}\right]\\
&\leq \bar{L}^{1/2}\,e^{-(\lambda-\bar{\alpha}/2)t}\,\widehat{\ew}_{(\mu_1,\mu_2)}\left[\mathbbm{1}_{\{\otau_n\leq t\}}e^{(\lambda-\bar{\alpha}/2)\otau_n} \bar{\rho}_{X,c}\left(\Phi_n^{(1)},\Phi_n^{(2)}\right)^{1/2}\right]\\
&\quad+\widehat{\ew}_{(\mu_1,\mu_2)}\left[\mathbbm{1}_{\{\widetilde{N}(t)=n\}}\mathbbm{1}_{\{\kappa>n\}} \right]\\
&\leq \bar{L}^{1/2}\,e^{-(\lambda-\bar{\alpha}/2)t}\,\left(\widehat{\ew}_{(\mu_1,\mu_2)}\left[\mathbbm{1}_{\{\otau_n\leq t\}}e^{(2\lambda-\bar{\alpha})\otau_n}\right]\right)^{1/2}\left(\widehat{\ew}_{(\mu_1,\mu_2)}\left[\bar{\rho}_{X,c}\left(\Phi_n^{(1)},\Phi_n^{(2)}\right)\right]\right)^{1/2}\\
&\,\quad+\widehat{\pr}_{(\mu_1,\mu_2)}(\widetilde{N}(t)=n)^{1/2}\,\widehat{\pr}_{(\mu_1,\mu_2)}(\kappa>n)^{1/2}.
\end{split}
\end{align}

What is left is to estimate the right-hand side of \eqref{mp:6}. To do this, we first observe that, for any $\lambda_0>0$,
\begin{align}
\begin{split}
\label{mp:7}
\widehat{\ew}_{(\mu_1,\mu_2)}\left[\mathbbm{1}_{\{\otau_n\leq t\}}e^{(2\lambda-\bar{\alpha})\otau_n}\right]&=\int_0^t e^{(2\lambda-\bar{\alpha})s} e^{-\lambda s} \frac{\lambda^n s^{n-1}}{(n-1)!}\,ds\\
&=\lambda_0\left(\frac{\lambda}{\lambda_0}\right)^n\int_0^t e^{(\lambda-\bar{\alpha})s} \frac{(\lambda_0 s)^{n-1}}{(n-1)!}\,ds\\
&\leq \lambda_0\left(\frac{\lambda}{\lambda_0}\right)^n\int_0^t e^{(\lambda-\bar{\alpha})s} e^{\lambda_0 s}\,ds=\lambda_0\left(\frac{\lambda}{\lambda_0}\right)^n\int_0^t e^{(\lambda+\lambda_0-\bar{\alpha})s} \,ds\\
&\leq \frac{\lambda_0}{\lambda+\lambda_0-\bar{\alpha}}\left(\frac{\lambda}{\lambda_0}\right)^n e^{(\lambda+\lambda_0-\bar{\alpha})t},
\end{split}
\end{align}
where the first equality follows from the fact that $\otau_n$ has the Erlang distribution with rate~$\lambda$.
Consequently, applying \eqref{mp:7}, together with hypothesis \eqref{e:coup_prop}, we see that
\begin{align}
\begin{split}
\label{mp:8}
&\bar{L}^{1/2}\,e^{-(\lambda-\bar{\alpha}/2)t}\,\left(\widehat{\ew}_{(\mu_1,\mu_2)}\left[\mathbbm{1}_{\{\otau_n\leq t\}}e^{(2\lambda-\bar{\alpha})\otau_n}\right]\right)^{1/2}\left(\widehat{\ew}_{(\mu_1,\mu_2)}\left[\bar{\rho}_{X,c}\left(\Phi_n^{(1)},\Phi_n^{(2)}\right)\right]\right)^{1/2}\\
&\leq C(\mu_1,\mu_2)^{1/2}\left(\frac{\bar{L}\lambda_0}{\lambda+\lambda_0-\bar{\alpha}}\right)^{1/2}\left(\frac{q\lambda}{\lambda_0}\right)^{n/2}e^{-(\lambda-\lambda_0)t/2}.
\end{split}
\end{align}
Moreover, from \eqref{mp:1} it follows that, for any $\bar{q}>0$,
\begin{align}
\begin{split}
\label{mp:9}
\widehat{\pr}_{(\mu_1,\mu_2)}(\widetilde{N}(t)=n)^{1/2}\,\widehat{\pr}_{(\mu_1,\mu_2)}(\kappa>n)^{1/2}
&\leq \left(e^{-\lambda t}\frac{(\lambda t)^n}{n!}\cdot C(\mu_1,\mu_2)q^n \right)^{1/2}\\
&=C(\mu_1,\mu_2)^{1/2}e^{-\lambda t/2}\left(\frac{(\lambda q\bar{q}^{\,-1}t)^n}{n!}\right)^{1/2}\bar{q}^{\,n/2}\\
&\leq C(\mu_1,\mu_2)^{1/2} e^{-\lambda t/2}e^{\lambda q\bar{q}^{\,-1}t/2}\,\bar{q}^{\,n/2}\\
&= C(\mu_1,\mu_2)^{1/2} e^{-\lambda(1-q\bar{q}^{\,-1})t/2}\,\bar{q}^{\,n/2}.
\end{split}
\end{align}

Let us now take $\bar{q}\in(q,1)$ and choose $\lambda_0\in (0,\lambda)$ so that $q\lambda\lambda_0^{-1}<1$. This choice guarantees that
$$\gamma:=\min\left\{\frac{\lambda-\lambda_0}{2}, \frac{\lambda(1-q\bar{q}^{\,-1})}{2} \right\}>0\quad\text{and}\quad r:=\max\left\{\left(\frac{q\lambda}{\lambda_0}\right)^{1/2},\, \bar{q}^{1/2} \right\}\in (0,1).$$
From \eqref{mp:6}, \eqref{mp:8} and \eqref{mp:9} we can now conclude that
\begin{align*}
\widehat{\ew}_{(\mu_1,\mu_2)}\left[\bar{\rho}_{X,c}\left(\Psi^{(1)}(t),\Psi^{(2)}(t)\right)\mathbbm{1}_{\{\widetilde{N}(t)=n\}}\right]\leq C(\mu_1,\mu_2)^{1/2}\left[\left(\frac{\bar{L}\lambda_0}{\lambda+\lambda_0-\alpha}\right)^{1/2}+1\right] r^n e^{-\gamma t}.
\end{align*}
Finally, defining $\widetilde{C}(\mu_1,\mu_2):=C(\mu_1,\mu_2)^{1/2}[(\bar{L}\lambda_0)^{1/2}(\lambda+\lambda_0-\alpha)^{-1/2} +1]$, we infer that
$$\widehat{\ew}_{(\mu_1,\mu_2)}\left[\bar{\rho}_{X,c}\left(\Psi^{(1)}(t),\Psi^{(2)}(t)\right)\right]\leq\sum_{n=0}^{\infty}\widetilde{C}(\mu_1,\mu_2)r^n e^{-\gamma t}=\frac{\widetilde{C}(\mu_1,\mu_2)}{1-r}e^{-\gamma t},$$
whence \eqref{erg_pt} holds with $\bar{C}(\mu_1,\mu_2):=\widetilde{C}(\mu_1,\mu_2)(1-r)^{-1}$. The proof is now complete.

\end{proof}

Lemmas \ref{lem:a}, \ref{lem:b} and Theorem \ref{thm:cor} enable us to prove the main result of this section, which reads as follows:
\begin{thm}\label{thm:main1}
Let $P$ be the transition operator of the chain $\Phi$, given by \eqref{def:P}, and let $\{P_t\}_{t\in\mathbb{R}_+}$ denote the transition semigroup of the process $\Psi$, defined by \eqref{def:pdmp}. Further, suppose that the following conditions are fulfilled:
\begin{enumerate}[label=\textnormal{(A\arabic*)}]
\item\label{cnd:a1}There exist $y^*\in Y$ and constants $\tilde{a}>0$, $\tilde{b}\geq 0$ for which $J$, occurring in \eqref{def:P_bar}, satisfies
\begin{equation}
\label{e:lap_p}
\p \rho_Y(\cdot,y^*)\leq \tilde{a}\rho_Y(\cdot, y^*)+\tilde{b}.
\end{equation}
\item\label{cnd:a2}We have
\begin{equation}
\label{e:gv}
\int_0^{\infty} e^{-\lambda t} \rho_Y(S_i(t,y^*),y^*)\,dt<\infty\quad\text{for any}\quad i\in I.
\end{equation}
\item\label{cnd:a3}There exist $L>0$ and $\alpha\in\mathbb{R}$ satisfying $\tilde{a}L+\alpha\lambda^{-1}<1$, for which \eqref{e:c} holds, that is,
$$\rho_Y(S_i(t,u), S_i(t,v))\leq Le^{\alpha t}\rho_Y(u,v)\;\;\;\mbox{for}\;\;\;u,v\in Y,\;i\in I,\;t\geq 0.$$ 
\end{enumerate}
\vspace{-0.5cm}
\begin{enumerate}[label=\textnormal{(C)}]
\item\label{cnd:c} There exists an augmented coupling $\widehat{\Phi}=\{(\Phi_n^{(1)},\Phi_n^{(2)}, \widetilde{\tau}_n)\}_{n\in\n_0}$ of the chain $\Phi$, with transition law $\widehat{P}$ satisfying \eqref{def:coup}, such that \eqref{e:coup_prop_v} holds with certain constants \hbox{$q\in (0,1)$}, $C_0<\infty$ and the function $V$ of the form
\begin{equation}
\label{V_rho}
V(y,i):=\rho_Y(y,y^*)\quad\text{for}\quad (y,i)\in X.\vspace{-0.3cm}
\end{equation}
\end{enumerate}
Then, if $J$ is Feller, \hbox{the~operator}~$P$ is \hbox{$V$-exponentially ergodic} in~$d_{FM,c}$  (in the sense of \hbox{Definition} \ref{def:erg_p}). Moreover, if \eqref{e:j_cont} holds, then the semigroup $\{P_t\}_{t\in\mathbb{R}_+}$ is \hbox{$V$-exponentially} ergodic in $d_{FM,c}$ (in the sense of~Definition~\ref{def:erg_pt}).
\end{thm}
\begin{proof}
First of all, note that conditions \ref{cnd:a1}-\ref{cnd:a3} imply that \eqref{e:lyapunov} holds with $V$ given by~\eqref{V_rho} and the constants 
\begin{equation}
\label{e:ab}
a:=\frac{\tilde{a}\lambda L}{\lambda -\alpha}\in(0,1)\quad\text{and}\quad b:=\tilde{a}\lambda\max_{i\in I}\int_0^{\infty} e^{-\lambda t} \rho_Y(S_i(t,y^*),y^*)\,dt+\tilde{b}\geq 0,
\end{equation}
where $\tilde{a}$, $\tilde{b}$ are determined by \ref{cnd:a1}. We are led to this conclusion by the following estimates:
\begin{align*}
PV(y,i)&=\int_0^{\infty}\lambda e^{-\lambda h}J\rho_Y(\cdot,y^*)(S_i(h,y))\,dh
\leq \tilde{a}\int_0^{\infty}\lambda e^{-\lambda h} \rho_Y(S_i(h,y),y^*)\,dh+\tilde{b}\\
&\leq  \tilde{a}\int_0^{\infty}\lambda e^{-\lambda h}\left[\rho_Y(S_i(h,y),S_i(h,y^*))+\rho_Y(S_i(h,y^*),y^*)\right]\,dh+\tilde{b}\\
& \leq  \tilde{a}\int_0^{\infty}\lambda e^{-\lambda h}\left[Le^{\alpha h}\rho_Y(y,y^*)+\rho(S_i(h,y^*),y^*)\right]\,dh+\tilde{b}\\
&= \tilde{a}\lambda L\left(\int_0^{\infty}e^{-(\lambda-\alpha)h} \,dh\right)V(y,i)+\tilde{a}\lambda\int_0^{\infty}e^{-\lambda h} \rho(S_i(h,y^*),y^*)\,dh+\tilde{b}\\
&\leq aV(y,i)+b\quad \text{for any}\quad y\in Y,\;i\in I.
\end{align*}

Consequently, if $J$ is Feller (and thus so is $P$, due to Remark \ref{rem:feller}), then, by virtue of Lemma~\ref{lem:a}, the operator $P$ is exponentially ergodic in $d_{FM,c}$. 

It now remains to prove that $\{P_t\}_{t\in\mathbb{R}_+}$ is also exponentially ergodic, provided that the Feller property of $J$ is strengthened to condition \eqref{e:j_cont}.

Let $\mu_*^{\Phi}\in\mathcal{M}_{prob}^V(X)$ be the unique invariant probability measure of $P$ (which exists by Lemma~\ref{lem:a}). Then, upon assuming \eqref{e:j_cont}, Theorem \ref{thm:cor} guarantees the existence of exactly one invariant probability measure for $\{P_t\}_{t\in\mathbb{R}_+}$, which can be expressed as $\mu_*^{\Psi}:=\mu_*^{\Phi}G$, where $G$ is the Markov operator induced by \eqref{def:g}. Moreover, conditions \ref{cnd:a2} and \ref{cnd:a3} yield that $\mu_*^{\Psi}\in\mathcal{M}_{prob}^V(X)$.

Now, to complete the proof, it suffices to show that there exists $\gamma>0$ such that
\begin{equation}
\label{e:pt_contr}
d_{FM,c}(\mu_1 P_t, \mu_2 P_t)\leq \bar{C}(\mu_1,\mu_2)e^{-\gamma t}\quad\text{for any}\quad t\geq 0, \; \mu_1,\mu_2\in\mathcal{M}_{prob}^V(X),
\end{equation}
where $\bar{C}(\mu_1,\mu_2)$ is a constant depending on $\mu_1$ and $\mu_2$.

From hypothesis \ref{cnd:c} and \eqref{e:p_m} it follows that
$$\widehat{\ew}_{(\mu_1,\mu_2)}\left[\bar{\rho}_{X,c}(\Phi_n^{(1)},\Phi_n^{(2)})\right]\leq C(\mu_1,\mu_2)q^n\quad\text{for any}\quad n\in\n,\;\mu_1,\mu_2\in\mathcal{M}_{prob}(X),$$
where 
$$C(\mu_1,\mu_2):=C_0(\<V,\mu_1\>+\<V,\mu_2\>+1)\quad\text{with some}\quad C_0\in\mathbb{R},$$ and, clearly, $C(\mu_1,\mu_2)<\infty$ for $\mu_1,\mu_2\in\mathcal{M}_1^V(X)$. In view of Lemma~\ref{lem:b}, this, together with condition \ref{cnd:a3}, implies the existence of $\gamma>0$ such that
$$
\widehat{\ew}_{(\mu_1,\mu_2)}\left[\bar{\rho}_{X,c}\left(\Psi^{(1)}(t),\Psi^{(2)}(t)\right)\right] \leq \bar{C}(\mu_1,\mu_2)\,e^{-\gamma t}\quad\text{for all}\quad t\geq 0,\;\mu_1,\mu_2\in\mathcal{M}_{prob}^V(X),
$$
where $\bar{C}(\mu_1,\mu_2)$ is a constant depending on $C(\mu_1,\mu_2)$. Now, it suffices to observe that, for any $\mu_1,\mu_2\in\mathcal{M}_1^V(X)$ and $f\in\Lip_{b,1}(X)$,
\begin{align*}
\left|\<f, \mu_1 P_t-\mu_2 P_t\>\right|&=\left|\widehat{\ew}_{(\mu_1,\mu_2)} \left[f\left(\Psi^{(1)}(t)\right)- f\left(\Psi^{(2)}(t)\right)\right] \right|\leq\widehat{\ew}_{(\mu_1,\mu_2)} \left| f\left(\Psi^{(1)}(t)\right)- f\left(\Psi^{(2)}(t)\right) \right|\\
&\leq \widehat{\ew}_{(\mu_1,\mu_2)}\left[\bar{\rho}_{X,c}\left(\Psi^{(1)}(t),\Psi^{(2)}(t) \right)\right]\leq \bar{C}(\mu_1,\mu_2)e^{-\gamma t},
\end{align*}
which obviously assures \eqref{e:pt_contr}, and thus ends the proof.
\end{proof}

\begin{rem}
It is worth noting that condition \eqref{e:pt_contr} is achieved by using only \ref{cnd:c} and~\eqref{e:c}.
\end{rem}


\section{Sufficient conditions for exponential ergodicity in~$d_{FM,c}$}\label{sec:main}
This section is intended to provide some verifiable, sufficient conditions for the existence of a suitable coupling of $\Phi$, for which condition \ref{cnd:c} of Theorem \ref{thm:main1} is satisfied, which, in turn, will enable us to state a verifiable criterion for the exponential ergodicity of $P$ and $\{P_t\}_{t\in\mathbb{R}_+}$. Obviously, such conditions should refer to the semiflows $S_i$, the probabilities $\pi_{ij}$ and the kernel $J$, appearing in the definition of $\bar{P}$, given in $\eqref{def:P_bar}$. 

To ensure the existence of an appropriate coupling, we need to assume that hypotheses \ref{cnd:a1}-\ref{cnd:a3} of Theorem \ref{thm:main1} hold and, additionally, employ the following conditions:
\begin{enumerate}[label=\textnormal{(A\arabic*)}, start=4]
\item\label{cnd:a4}There exist a Lebesgue measurable function $\varphi:\mathbb{R}_+\to\mathbb{R}_+$ satisfying \vspace{-0.1cm}
$$
K_{\varphi}:=\int_0^{\infty}e^{-\lambda t} \varphi(t)\,dt<\infty
\vspace{-0.1cm}
$$ 
and a function $\mathcal{L}:Y\to\mathbb{R}_+$ that is bounded on bounded sets such that
$$\rho_Y(S_i(t,y),S_j(t,y))\leq \varphi(t)\mathcal{L}(y)\quad\text{for any}\quad t\geq 0,\;y\in Y,\;i,j\in I.$$
\item\label{cnd:a5}There exists $j_0\in I$ such that $\min_{i\in I}\pi_{ij_0}>0$.
\item\label{cnd:a6}There exists a substochastic kernel $Q_J: Y^2\times\mathcal{B}(Y^2)\to [0,1]$ such that
\begin{equation}
\label{e:Q_J}
Q_J((y_1,y_2), B\times Y)\leq J(y_1,B)\quad\text{and}\quad Q_J((y_1,y_2), Y\times B)\leq J(y_2,B)
\end{equation}
for any $y_1,y_2\in Y$ and $B\in\mathcal{B}(Y)$, which enjoys the following properties:
\begin{equation}
\label{e:q1} \int_{Y^2}\rho_Y(u,v)\,Q_J((y_1,y_2),du\times dv)\leq \tilde{a}\rho_Y(y_1,y_2) \quad \text{for any}\quad y_1,y_2\in Y,
\end{equation}
\begin{equation}
\label{e:q2} \inf_{(y_1,y_2)\in Y^2} Q_J((y_1,y_2), \widetilde{U}\left(\tilde{a}\rho_Y(y_1,y_2)\right)\geq \eta\quad\text{for some}\quad \eta>0,\vspace{-0.1cm}
\end{equation}
where
\begin{equation}
\label{def:u_tilde}
\widetilde{U}(r):=\{(u,v)\in Y^2:\,\rho_Y(u,v)\leq r\}\quad\text{for} \quad r>0,
\end{equation}
and $\tilde{a}$ is the constant for which \ref{cnd:a1} holds, as well as there exists $\tilde{l}>0$ such that
\begin{equation}
\label{e:q3} Q_J((y_1,y_2), Y^2)>1-\tilde{l}\rho_Y(y_1,y_2)\quad\text{for any}\;\; y_1,y_2\in Y.
\end{equation}
\end{enumerate}\vspace{-0.2cm}

Furthermore, we will assume that the constant $c$, appearing in \eqref{def:rho_c}, is sufficiently large. More specifically, we shall require that
\begin{equation}
\label{e:const_c}
c\geq \frac{\lambda-\alpha}{L}\left(M_{\mathcal{L}}{K}_{\varphi}+\frac{M_{\mathcal{L}}M_{\varphi}}{\lambda}\right)+1,\vspace{-0.3cm}
\end{equation}
where\vspace{-0.1cm}
\begin{gather}
\label{def:M_l} M_{\mathcal{L}}:=\sup\{\mathcal{L}(y):\,\rho_Y(y,y^*)\leq 4b/(1-a)\},\\
\label{def:M_f} M_{\varphi}:=\sup\left\{\varphi(t): t\leq \lim_{s\to \alpha} s^{-1}\ln\left( \lambda(\lambda-s)^{-1}\right)\right\},
\end{gather}
and the constants $a$ and $b$ are given by \eqref{e:ab}.

Conditions \ref{cnd:a3} and \ref{cnd:a4} are fulfilled, e.g. for the flows generated by some classes of dissipative differential equations. This rests on the following observation: 
\begin{rem}
Suppose that $Y$ is a closed subset of a Hilbert space $H$, endowed with an inner product $\<\cdot|\cdot\>$. Let $A_i: Y\to H$, $i\in I$, be a finite collection of $\alpha$-dissipative operators with some $\alpha\leq 0$, i.e.
$$\<A_i y_1 - A_i y_2\,|\,y_1-y_2\>\leq \alpha \norma{y_1-y_2}^2\quad\text{for any}\quad y_1,y_2\in Y,\;i\in I.$$
Furthermore, assume that there exists $T\in (0,\infty)$, such that
$$Y\subset\operatorname{Range}(\operatorname{id}_Y-tA_i)\quad\text{for all}\quad t\in (0,T),\;i\in I.$$
Then, according to \cite[Theorem 5.11]{b:kappel}, for any $i\in I$ and $y\in Y$, the initial value problem
$$\frac{d}{dt}u(t)=A_i\,u(t)\;\;\text{for}\;\;t\geq 0,\quad u(0)=y,$$
has a unique (strong) solution $\mathbb{R}_+\ni t\mapsto S_i(t,y)\in Y$, which obviously generates a semiflow. What is more, by virtue of \hbox{\cite[Theorem 5.3 and Corollary 5.4]{b:kappel}}, the semiflows $S_i$ satisfy 
$$\norma{S_i(t,y_1)-S_i(t,y_2)}\leq e^{\alpha t}\norma{y_1-y_2}\quad\text{for any}\quad y_1,y_2\in Y,$$
$$\norma{S_i(t,y)-y}\leq t \norma{A_iy}\quad\text{for any}\quad y\in Y.$$
This, in turn, implies that conditions \ref{cnd:a3} and \ref{cnd:a4} hold for such $S_i$ with the given $\alpha$,
$$L=1,\quad\mathcal{L}(y)=2\max_{i\in I}\norma{A_i y}\quad
\text{and}\quad\varphi(t)=t,$$
provided that $A_i$, $i\in I$, are bounded on bounded sets, and that $\widetilde{a}+\alpha/\lambda<1$ for $\widetilde{a}$ specified by \ref{cnd:a1}.
\end{rem}

The following simple example (inspired by \cite[Example 5.2]{b:benaim1}) illustrates the case wherein \ref{cnd:a4} holds with a non-linear function $\varphi$.
\begin{ex}
Suppose that $Y$ is a (closed) subset of a Banach space, and consider the semiflows $S_1,S_2:\mathbb{R}_+\times Y\to Y$ given by
$$S_1(t,y)=e^{\alpha t}y,\quad S_2(t,y)=e^{\alpha t}(y-r)+r,$$
where $r\in\mathbb{R}\backslash\{0\}$ and $\alpha<0$. Then conditions \ref{cnd:a3} and \ref{cnd:a4} hold for $\{S_1,S_2\}$ with the given $\alpha$, $L=1$, $\mathcal{L}\equiv 1$ and $\varphi(t)=|r|(1-e^{\alpha t})$, provided that $\widetilde{a}+\alpha/\lambda<1$.
\end{ex}

\subsection{The main result}
Let us consider a substochastic kernel $\bar{Q}_P:Z\times \mathcal{B}(Z)\to [0,1]$ (where $Z=X^2\times\mathbb{R}_+$), given~by
\begin{align}
\begin{split}
\label{def:q_bar}
\bar{Q}_P((x_1,x_2,s),D)&:=\sum_{j\in I} (\pi_{i_1,j} \wedge \pi_{i_2,j})\int_0^{\infty}\lambda e^{-\lambda h} \int_{Y^2} \mathbbm{1}_D((u_1,j),(u_2,j),h+s)\\&\quad\times Q_J((S_{i_1}(h,y_1),S_{i_2}(h,y_2)),du_1\times du_2)\,dh
\end{split}
\end{align}
for any $x_1:=(y_1,i_1), x_2:=(y_2,i_2)\in X$, $s\in\mathbb{R}_+$ and $D\in\mathcal{B}(Z)$. 

Having in mind \eqref{e:Q_J}, it is easy to see that, for any $x_1,x_2\in X$, $s\in\mathbb{R}_+$, $A\in\mathcal{B}(X)$ and $T\in\mathcal{B}(\mathbb{R}_+)$, we have
\begin{equation}
\label{def:r_bar}
\begin{gathered}
\bar{Q}_P((x_1,x_2,s),A\times X\times\mathbb{R}_+)\leq P(x_1,A), \\ \bar{Q}_P((x_1,x_2,s),X\times A\times\mathbb{R}_+)\leq P(x_2,A),\\
\bar{Q}_P((x_1,x_2,s),X^2\times T)\leq E_{\lambda}(s,T),
\end{gathered}
\end{equation}
where $E_{\lambda}(s,\cdot)$ denotes the distribution with density $t\mapsto \mathbbm{1}_{[s,\infty)}(t)\lambda e^{\lambda(t-s)}$. This enables us to define a substochastic kernel $\bar{R}_P: Z\times \mathcal{B}(Z)\to [0,1]$ so that, on cubes $D:=A_1\times A_2\times T$, where $A_1,A_2\in\mathcal{B}(X)$ and $T\in\mathcal{B}(\mathbb{R}_+)$, the measure $\bar{R}_P((x_1,x_2,s),\cdot)$ is given by
\begin{align}
\label{e:qr}
\bar{R}_P((x_1,x_2,s),D)&:=\frac{1}{\left[1-\bar{Q}_P((x_1,x_2,s),Z)\right]^2}  \left[P(x_1,A_1)-\bar{Q}_P((x_1,x_2,s),A_1\times X\times\mathbb{R}_+)\right]\nonumber\\
&\;\quad\times \left[P(x_2,A_2)-\bar{Q}_P((x_1,x_2,s),X\times A_2\times\mathbb{R}_+)\right]\\
&\;\quad\times \left[E_{\lambda}(s,T)-\bar{Q}_P((x_1,x_2,s),X^2\times T)\right]\nonumber
\end{align}
when $\bar{Q}_P((x_1,x_2,s),Z)<1$, and $\bar{R}_P((x_1,x_2,s),D):=0$ otherwise. 

A simple computation shows that $\widehat{P}:\mathcal{B}(Z)\times Z\to[0,1]$ given by
\begin{equation}
\label{e:coup_exp}
\widehat{P}((x_1,x_2,s),D):=\bar{Q}_P((x_1,x_2,s),D)+\bar{R}_P((x_1,x_2,s),D)
\end{equation}
for any $x_1,x_2\in X,\;s\geq 0$ and $D\in\mathcal{B}(Z)$ defines a stochastic kernel satisfying conditions~\eqref{def:coup}. In~other words, the kernel defined in this way can play the role of transition law of the augmented coupling \hbox{$\widehat{\Phi}:=\{(\Phi_n^{(1)},\Phi_n^{(2)}, \widetilde{\tau}_n)\}_{n\in\n_0}$} discussed in Section \ref{sec:coup}. 

What is more, one can show that such a coupling fulfils hypothesis \ref{cnd:c} of Theorem \ref{thm:main1}, which is stated precisely in the following result:
\begin{prop}\label{lem:coup}
Suppose that conditions \ref{cnd:a1}-\ref{cnd:a6} and \eqref{e:const_c} hold. Then the coupling $\widehat{\Phi}$  with transition law $\widehat{P}$ defined by \eqref{e:coup_exp} satisfies \eqref{e:coup_prop_v} with $V$ given by \eqref{V_rho}, some $q\in (0,1)$ and some $C_0<\infty$.
\end{prop}
The proof of this statement proceeds almost in the same way as that of \cite[Lemma 2.3]{b:czapla_clt}, provided that hypotheses \ref{cnd:b1}-\ref{cnd:b5} stated in \cite[Section 2]{b:czapla_clt} are fulfilled for the operator $P$, given by \eqref{def:P}, and the kernel $Q_P:X^2\times\mathcal{B}(X^2)\to [0,1]$ of the form
\begin{equation}
\label{def:q}
Q_P((x_1,x_2),C):=\bar{Q}_P((x_1,x_2,0),C\times\mathbb{R}_+),\quad x_1,x_2\in X,\;C\in\mathcal{B}(X^2),
\end{equation}
where $\bar{Q}_P$ is defined by \eqref{def:q_bar}. These hypotheses (also assumed in \cite[Theroem 2.1]{b:kapica}) can be derived quite easily from conditions \ref{cnd:a1}-\ref{cnd:a6}. The proof of this claim, as well as a suitable adaptation of the reasoning employed in~\cite{b:czapla_clt}, which eventually proves Proposition \ref{lem:coup}, are postponed to Section \ref{proof:lem}.

In view of Proposition \ref{lem:coup}, we can replace hypothesis \ref{cnd:c} of Theorem~\ref{thm:main1} with conditions \ref{cnd:a4}-\ref{cnd:a6}, which, together with \ref{cnd:a1}-\ref{cnd:a3}, guarantee the existence of a suitable coupling of~$\Phi$. This leads us to the main result of the paper:
\begin{thm}\label{thm:main2}
Suppose that conditions \ref{cnd:a1}-\ref{cnd:a6} and \eqref{e:const_c} hold. Further, let $V$ be given by~\eqref{V_rho}. Then, if $J$ is Feller, the transition operator $P$ of the chain $\Phi$, induced by \eqref{def:P}, is $V$-exponentially ergodic in~$d_{FM,c}$. Moreover, if \eqref{e:j_cont} holds, then the transition semigroup $\{P_t\}_{t\in\mathbb{R}_+}$ of the process $\Psi$, defined by \eqref{def:pdmp}, is \hbox{$V$-exponentially} ergodic in $d_{FM,c}$ as well.
\end{thm}


\subsection{A model with jumps generated by random iterations}\label{sec:ifs}
Let us look closer at the case that has already been mentioned in Remark \ref{rem:spec}. For simplicity of notation, we will skip the perturbations (the linear structure of $Y$ is then not required). In such a case, the kernel $J$ is the transition law of an iterated function system, consisting of an arbitrary set \hbox{$\mathcal{W}:=\{w_{\theta}:\,\theta\in \Theta\}$} of continuous transformations from $Y$ to itself and an associated set $\mathcal{P}:=\{p_{\theta}:\,\theta\in\Theta\}$ of place-dependent probabilities, mapping $Y$ to $[0,1]$. Here, it is assumed that $(\Theta,\mathcal{B}(\Theta),\vartheta)$ is a topological space with a measure $\vartheta$, the maps $(y,\theta)\mapsto w_{\theta}(y)$, $(y,\theta)\mapsto p_{\theta}(y)$ are product measurable, and  $\int_{\Theta} p_{\theta}(y)\,\vartheta(d\theta)=1$ for any~$y\in Y$. 

In the above-described setting, $J$ is given by
\begin{equation}
\label{e:ifs}
\p(y,B)=\int_{\Theta}\mathbbm{1}_B(w_{\theta}(y))\,p_{\theta}(y)\,\vartheta(d\theta)\quad\text{for}\quad y\in Y,\; B\in\mathcal{B}(Y),
\end{equation}
and $P$ takes the form
\begin{equation}
\label{def:P_ifs}
P((y,i), A)=\sum_{j\in I} \pi_{ij} \int_0^{\infty} \lambda e^{-\lambda h}\int_{\Theta} \mathbbm{1}_{A}(w_{\theta}(S_i(h,y),j)\,p_{\theta}(S_i(h,y))\,\vartheta(d\theta)\,dh
\end{equation}
for any $y\in Y$, $i\in I$ and $A\in\mathcal{B}(X)$. Moreover, note that, in this framework, the first coordinate of the chain~$\Phi$ can be expressed explicitly by the recursive formula:
$$Y_{n+1}=w_{\theta_{n+1}}(S_{\xi_n}(\Delta\tau_{n+1},Y_n)),\quad n\in\n_0,$$
where $\{\theta_n\}_{n\in\n}$ is an appropriate sequence of random variables with values in $\Theta$, such that
$$\pr_{\nu}(\theta_{n+1}\in D\,|\,S_{\xi_n}(\Delta\tau_{n+1},Y_n)=y)=\int_D p_{\theta}(y)\vartheta(d\theta)\quad\text{for}\quad D\in\mathcal{B}(\Theta),\;y\in Y,\;n\in\n.$$

We shall impose the following assumptions (in the spirit of those made in \hbox{\cite[Proposition 3.1]{b:kapica}}; cf. also \cite{b:sleczka} and \hbox{\cite[Theorem 3.1]{b:szarek}}) on the system $(\mathcal{W},\mathcal{P})$: there exist $y^*\in Y$, for which
\begin{equation}
\label{cnd:i1} \tilde{b}:=\sup_{y\in Y}\int_{\Theta} \rho_Y(w_{\theta}(y^*),y^*)p_{\theta}(y)\,\vartheta(d\theta)<\infty,
\end{equation}
and positive constants $\tilde{a}, \tilde{l}$ and $\eta$ such that, for any $y_1,y_2\in Y$,
\begin{gather}
\label{cnd:i2} \int_{\Theta} \rho_Y(w_{\theta}(y_1),w_{\theta}(y_2))\,p_{\theta}(y_1)\vartheta(d\theta)\leq \tilde{a}\rho_Y(y_1,y_2),\\
\label{cnd:i4} \int_{\Theta(y_1,y_2)} p_{\theta}(y_1)\wedge p_{\theta}(y_2)\,\vartheta(d\theta)\geq \eta,
\end{gather}
where
\begin{equation*}
\Theta(y_1,y_2):=\{\theta\in \Theta:\,\rho_Y(w_{\theta}(y_1),w_{\theta}(y_2))\leq \tilde{a}\rho_Y(y_1,y_2)\},
\end{equation*}
and
\begin{equation}
\label{cnd:i3} \int_{\Theta} |p_{\theta}(y_1)-p_{\theta}(y_2)|\,\vartheta(d\theta)\leq \tilde{l}\rho_Y(y_1,y_2).
\end{equation}
\begin{rem}
Note that \eqref{cnd:i1} is trivially satisfied in the case where $\Theta$ is compact, and $\theta\mapsto w_{\theta}(y^*)$ is continuous for some $y^*\in Y$.
\end{rem}

Theorem \ref{thm:main2} allows us to establish the following result:
\begin{prop}\label{prop:ifs}
Suppose that the kernel $J$ is of the form \eqref{e:ifs}, and the transformations~$w_{\theta}$ are continuous. Further, assume that there exist $y^*\in Y$ and positive constants $\tilde{a}, \tilde{l}$, $\eta$ such that conditions \eqref{cnd:i1}-\eqref{cnd:i3}, \ref{cnd:a2}-\ref{cnd:a5} and \eqref{e:const_c} hold. Then both the transition operator $P$ of the chain $\Phi$ (induced by \eqref{def:P_ifs} in this case) and the transition semigroup $\{P_t\}_{t\in\mathbb{R}_+}$ of the process $\Psi$, defined by \eqref{def:pdmp}, are $V$-exponentially ergodic in $d_{FM,c}$ with $V$ given by \eqref{V_rho}.
\end{prop}
\begin{proof}
In view of Theorem \ref{thm:main2}, it suffices to show that conditions  \ref{cnd:a1}, \ref{cnd:a6} and \eqref{e:j_cont} hold.

First of all, note that \ref{cnd:a1} follows immediately from \eqref{cnd:i1} and \eqref{cnd:i2}, since
$$J\rho_Y(\cdot,y^*)(y)=\int_{\Theta} \rho_Y(w_{\theta}(y),y^*)p_{\theta}(y)\,\vartheta(d\theta)\leq \tilde{a}\rho_Y(y^*,y)+\tilde{b}\quad\text{for all}\quad y\in Y.$$

Now, we will show that \ref{cnd:a6} is fulfilled with $Q_J:X^2\times\mathcal{B}(X^2)\to[0,1]$ given by
\begin{equation}
\label{def:Qj_ifs}
Q_J((y_1,y_2),C):=\int_{\Theta}\mathbbm{1}_C(w_{\theta}(y_1),w_{\theta}(y_2))(p_{\theta}(y_1)\wedge p_{\theta}(y_2))\,\vartheta(d\theta)
\end{equation}
for any $y_1,y_2\in Y$ and $C\in\mathcal{B}(Y^2)$. Obviously, $Q_J$ is a substochastic kernel satisfying~\eqref{e:Q_J}. Condition \eqref{cnd:i2} yields that, for any $y_1,y_2\in Y$,
$$Q_J\rho_Y(y_1,y_2)=\int_{\Theta}\rho_Y(w_{\theta}(y_1),w_{\theta}(y_2))(p_{\theta}(y_1)\wedge p_{\theta}(y_2))\,\vartheta(d\theta)\leq \tilde{a}\rho_Y(y_1,y_2),$$
which gives \eqref{e:q1}. Further, \eqref{cnd:i4} implies \eqref{e:q2}, since, for any $y_1,y_2\in Y$, we have
$$Q_J((y_1,y_2),\widetilde{U}(\tilde{a}\rho_Y(y_1,y_2))=\int_{\Theta(y_1,y_2)} p_{\theta}(y_1)\wedge p_{\theta}(y_2)\,\vartheta(d\theta)\geq \eta>0,$$
with $\widetilde{U}(\cdot)$ defined by \eqref{def:u_tilde}. Finally, \eqref{e:q3} can be easily concluded from hypothesis \eqref{cnd:i3} and the \hbox{inequality} $s\wedge t \geq s-|s-t|$, which is valid for any $s,t\in\mathbb{R}$.

What is left is to show that $J$ satisfies \eqref{e:j_cont}. To this end, let $g\in C_b(Y\times\mathbb{R}_+)$ and fix \hbox{$(y_0,t_0)\in Y\times\mathbb{R}_+$}. Then, again using \eqref{cnd:i3}, we see that
\begin{align*}
|Jg(\cdot,t_0)(y_0) - Jg(\cdot,t)(y)|&\leq\int_{\Theta}|g(w_{\theta}(y_0),t_0)p_{\theta}(y_0)-g(w_{\theta}(y),t)p_{\theta}(y)|\,\vartheta(d\theta)\\
&\leq\int_{\Theta}|g(w_{\theta}(y_0),t_0)-g(w_{\theta}(y),t)|p_{\theta}(y_0)\,\vartheta(d\theta)\\
&\quad + \norma{g}_{\infty}\tilde{l}\rho_Y(y_0,y)\quad\text{for any}\quad (y,t)\in Y\times\mathbb{R}_+.
\end{align*}
Consequently, having in mind the continuity of $g$ and the transformations $w_\theta$, we conclude that the map $Y\times\mathbb{R}_+\ni (y,t)\mapsto Jg(\cdot,t)(y)$ is jointly continuous, by applying the Lebesgue dominated convergence theorem. The use of Theorem \ref{thm:main2} now ends the proof.
\end{proof}

\begin{rem}
Obviously, Proposition \ref{prop:ifs} remains valid if the kernel $J$ is defined exactly as in Remark~\ref{rem:spec}. The proof is then almost the same as that given above. In that case, however, one needs to consider $Q_J$ of the form
$$Q_J((y_1,y_2),C):=\int_{\supp\nu}\int_{\Theta}\mathbbm{1}_C(w_{\theta}(y_1)+v,w_{\theta}(y_2)+v) (p_{\theta}(y_1)\wedge p_{\theta}(y_2))\,\vartheta(d\theta)\,\nu(dv).$$
\end{rem}

\subsection{Proof of Proposition \ref{lem:coup}} \label{proof:lem}
Let $P$ and $Q_P$ be the kernels defined by \eqref{def:P} and \eqref{def:q}, respectively. Moreover, consider the augmented coupling $\widehat{\Phi}$ of the chain $\Phi$ (constructed in Section \ref{sec:coup}) with transition law $\widehat{P}$ defined by \eqref{e:coup_exp}. In particular, $\{(\Phi_n^{(1)},\Phi_n^{(2)} )\}_{n\in\n_0}$ itself is then governed by the kernel
\begin{equation}
\label{def:P_tilde}
\widetilde{P}((x_1,x_2),C):=\widehat{P}((x_1,x_2,0),C\times \mathbb{R}_+)\quad\text{for}\quad x_1,x_2\in X,\;C\in\mathcal{B}(X^2).
\end{equation}

In order to prove Proposition~\ref{lem:coup}, we first need to derive hypotheses \ref{cnd:b1}-\ref{cnd:b5}, used in \hbox{\cite[Section 2]{b:czapla_clt}} (and also assumed in \cite[Theroem 2.1]{b:kapica}), from conditions \ref{cnd:a1}-\ref{cnd:a6} and \eqref{e:const_c}. To begin, let us recall that $a$, $b$ are the constants specifed by \eqref{e:ab}, and that $V$ stands for the function given by \eqref{V_rho}. Further, define
$$F:=G\cup K,\vspace{-0.4cm}$$ where
\begin{equation}
\begin{gathered}
\label{def:F}
G:=\{((y_1,i_1),(y_2,i_2))\in X^2:\,i_1=i_2\},\\
K:=\{((y_1,i_1),(y_2,i_2))\in X^2:\, V(y_1,i_1)+V(y_2,i_2)<R \}\quad\text{with}\quad R:=\frac{4b}{1-a}.
\end{gathered}
\end{equation}
\begin{lem} 
\label{lem:help} 
Suppose that conditions \ref{cnd:a1}-\ref{cnd:a6} and \eqref{e:const_c} hold. Then the following statements are fulfilled:
\begin{enumerate}[label=\textnormal{(B\arabic*)}]
\item\label{cnd:b1}$PV(x)\leq aV(x)+b\;\;\;\mbox{for any}\;\;\;x\in X.$
\item\label{cnd:b2}$\supp Q_P((x_1,x_2),\cdot)\subset F$ and
$$\int_{X^2}\rho_{X,c}(w_1,w_2)\, Q_P((x_1,x_2),dw_1\times dw_2)\leq a\rho_{X,c}(x_1,x_2)\;\;\;\mbox{for any}\;\;\;(x_1,x_2)\in F.$$
\item\label{cnd:b3}Defining $U(r):=\{(w_1,w_2)\in X^2:\, \rho_{X,c}(w_1,w_2)\leq r\}$ for any $r>0$, we have
$$\inf_{(x_2,x_2)\in F}Q_P((x_1,x_2),U(a\rho_{X,c}(x_1,x_2)))>0,$$
\item\label{cnd:b4}There exists  $l>0$ such that $Q_P((x_1,x_2),X^2)\geq 1-l\rho_{X,c}(x_1,x_2)$ for any $(x_1,x_2)\in F$.
\item\label{cnd:b5}There exist $\gamma\in(0,1)$ and $C_{\gamma}>0$ such that
$$\widehat{\ew}_{(x_1,x_2)}(\gamma^{-\sigma_K})\leq C_{\gamma}\quad\mbox{whenever}\;\;V(x_1)+V(x_2)<R,$$
where $\sigma_K:=\inf\{n\in\n:\,(\Phi^{(1)}_n,\Phi^{(2)}_n)\in K\}.$
\end{enumerate}
\end{lem}
\begin{proof}
First of all, note that condition \ref{cnd:b1} has already been established at the beginning of the proof of Theorem \ref{thm:main1}.

The proof of the first part of \ref{cnd:b2} goes as follows. Let $(x_1,x_2):=((y_1,i_1),(y_2,i_2))\in X^2$. Since $X^2$ is endowed with the product topology, we may consider it as a metric space with the distance
$$\rho_{X^2,c}((w_1,w_2),(z_1,z_2)):=\rho_{X,c}(w_1,z_1)+\rho_{X,c}(w_2,z_2)\quad\text{for}\quad (w_1,w_2),(x_1,x_2)\in X^2.$$
The support of $Q_P(x_1,x_2,\cdot)$ can be then expressed as
$$\supp Q_P((x_1,x_2),\cdot)=\{(z_1,z_2)\in X^2:\, Q_P((x_1,x_2),B_{X^2}((z_1,z_2),\varepsilon))>0 \;\;\text{for any}\;\;\varepsilon>0\},$$
where $B_{X^2}((z_1,z_2),\varepsilon)$  is the open ball in $(X^2,\rho_{X^2,c})$ centered at $(z_1,z_2)$ with radius $\varepsilon$. Let 
$$(z_1,z_2):=((u_1,j_1),(u_2,j_2))\in X^2\backslash F.$$
Then, in particular, $(z_1,z_2)\notin G$, and thus $j_1\neq j_2$. This implies that, for any point $(w_1,w_2):=((v_1,j),(v_2,j))\in G$, we have
$$\rho_{X^2,c}((w_1,w_2),(z_1,z_2))\geq c(\dd(j,j_1)+\dd(j,j_2))\geq c,$$
whence $B_{X^2}((z_1,z_2),c))\cap G=\emptyset$. Taking into account the definition of $Q_P$, given in \eqref{def:q}, we therefore obtain
$$Q_P((x_1,x_2), B_{X^2}((z_1,z_2),c))=Q_P((x_1,x_2), B_{X^2}((z_1,z_2),c)\cap G)=0,$$
which yields that $(z_1,z_2)\in X^2\backslash \supp Q_P((x_1,x_2),\cdot)$.

Passing to the proof of the second part of \ref{cnd:b2}, let $(x_1,x_2)=((y_1,i_1),(y_2,i_2))\in F$. Then $i_1=i_2$ or $y_1,y_2\in B_Y(y^*,R)$ (due to the definition of $V$). Hence, from \ref{cnd:a3} and \ref{cnd:a4} it follows that
\begin{align}
\begin{split}
\label{e:est_s}
\rho_Y(S_{i_1}(t,y_1),S_{i_2}(t,y_2))&\leq \rho_Y(S_{i_1}(t,y_1),S_{i_1}(t,y_2))+\rho_Y(S_{i_1}(t,y_2),S_{i_2}(t,y_2))\\
&\leq Le^{\alpha t}\rho_Y(y_1,y_2)+\varphi(t)\mathcal{L}(y_2)\dd(i_1,i_2)\\
&\leq Le^{\alpha t}\rho_Y(y_1,y_2)+\varphi(t)M_{\mathcal{L}}\dd(i_1,i_2),
\end{split}
\end{align}
where $M_{\mathcal{L}}$ is given by \eqref{def:M_l}. Consequently, referring to the definition of $Q_P$, \eqref{e:q1} and \eqref{e:const_c}, we can conclude that
\begin{align*}
&\int_{X^2}\rho_{X,c}(w_1,w_2)\, Q_P((x_1,x_2),dw_1\times w_2)\\
&=\sum_{j\in I} (\pi_{i_1,j}\wedge \pi_{i_2,j})\int_0^{\infty}\lambda e^{-\lambda h}\int_{Y^2}\rho_Y(v_1,v_2)\,Q_J((S_{i_1}(h,y_1),S_{i_2}(h,y_2)),dv_1\times dv_2)\,dh\\
&\leq  \tilde{a}\lambda \int_0^{\infty}e^{-\lambda h}\rho_Y(S_{i_1}(h,y_1),S_{i_2}(h,y_2))\,dh\\
&\leq \tilde{a}\lambda L\left(\int_0^{\infty} e^{-(\lambda-\alpha)h}\,dh \right)\rho_Y(y_1,y_2)+\tilde{a}\lambda M_{\mathcal{L}}\left(\int_0^{\infty}e^{-\lambda h}\varphi(h)\,dh\right)\dd(i_1,i_2)\\
&=\frac{\tilde{a}\lambda L}{\lambda-\alpha}\rho_Y(y_1,y_2)+ \frac{\tilde{a}\lambda L}{\lambda-\alpha}\frac{(\lambda-\alpha)M_{\mathcal{L}}K_{\varphi}}{L}\,\dd(i_1,i_2)\leq a(\rho_Y(y_1,y_2)+c\dd(i_1,i_2))\\
&=a\cdot\rho_{X,c}(x_1,x_2),
\end{align*} 
which is the desired claim.

We now proceed to show condition \ref{cnd:b3}. First, define $t_0:=\lim_{s\to \alpha} s^{-1}\ln\left( \lambda(\lambda-s)^{-1}\right)$, which is obviously positive, and observe that
\begin{equation}
\label{e:est_T}
\tilde{a}Le^{\alpha t}\leq\frac{\tilde{a}\lambda L}{\lambda-\alpha}=a\quad\text{for any}\quad t\leq t_0.
\end{equation}
Now, let $(x_1,x_2)=((y_1,i_1),(y_2,i_2))\in F$, and note that, for any $u_1,u_2\in Y$, $j\in I$, and $0\leq t\leq t_0$, we have
\begin{equation}
\label{e:imp_u}
(u_1,u_2)\in \widetilde{U}(\tilde{a}\rho_Y(S_{i_1}(t,y_1),S_{i_2}(t,y_2))\;\;\Rightarrow\;\; ((u_1,j),(u_2,j))\in U(a\rho_{X,c}(x_1,x_2)),
\end{equation}
where $\widetilde{U}(\cdot)$ and $U(\cdot)$ are defined as in  \eqref{def:u_tilde} and \ref{cnd:b3}, respectively. To see this, it suffices to \hbox{apply} \eqref{e:est_s}, \eqref{e:est_T} and \eqref{e:const_c}, which ensure that, for any $(u_1,u_2)\in U(\tilde{a}\rho_Y(S_{i_1}(t,y_1),S_{i_2}(t,y_2))$ with $t\leq t_0$,
\begin{align*}
\rho_{X,c}((u_1,j),(u_2,j))&=\rho_Y(u_1,u_2)\leq \tilde{a}\rho_Y(S_{i_1}(t,y_1),S_{i_2}(t,y_2))\\
&\leq \tilde{a}Le^{\alpha t}\rho_Y(y_1,y_2)+\tilde{a}M_{\mathcal{L}}\varphi(t)\dd(i_1,i_2)\\
&\leq a\rho_Y(y_1,y_2)+\tilde{a}M_{\mathcal{L}}M_{\varphi}\dd(i_1,i_2)\\
&= a\left(\rho_Y(y_1,y_2)+\frac{M_{\mathcal{L}}M_{\varphi}(\lambda-\alpha)}{\lambda L}\dd(i_1,i_2)\right)\leq a \rho_{X,c}(x_1,x_2),
\end{align*}
with $M_{\varphi}$ given by \eqref{def:M_f}, whence $((u_1,j),(u_2,j))\in U(a\rho_{X,c}(x_1,x_2))$. Now, using \eqref{e:imp_u}, together with \eqref{e:q2}, we obtain
\begin{align*}
&\int_{Y^2}\mathbbm{1}_{U(a\rho_{X,c}(x_1,x_2))}((u_1,j),(u_2,j)) Q_J((S_{i_1}(h,y_1),S_{i_2}(h,y_2)),du_1\times du_2)\\
&\geq \int_{Y^2}\mathbbm{1}_{\widetilde{U}\left(\tilde{a}\rho_Y(S_{i_1}(h,y_1),S_{i_2}(h,y_2)\right)}(u_1,u_2)Q_J((S_{i_1}(h,y_1),S_{i_2}(h,y_2)),du_1\times du_2)\\
&= Q_J((S_{i_1}(h,y_1),S_{i_2}(h,y_2)),\widetilde{U}(\tilde{a}\rho_Y(S_{i_1}(h,y_1),S_{i_2}(h,y_2)))\geq \eta\quad\text{for any}\quad h\leq t_0.
\end{align*}
Finally, from \ref{cnd:a5} it follows that
\begin{align*}
Q_P((x_1,x_2),U(a\rho_{X,c}(x_1,x_2)))&\geq\sum_{j\in I} (\pi_{i_1,j}\wedge\pi_{i_2,j})\int_0^{t_0} \lambda e^{-\lambda h}\int_{Y^2}\mathbbm{1}_{U(a\rho_{X,c}(x_1,x_2))}((u_1,j),(u_2,j))\\
&\quad\times Q_J((S_{i_1}(h,y_1),S_{i_2}(h,y_2)),du_1\times du_2)\,dh\\
&\geq (\min_{i\in I}\pi_{ij_0})\eta(1-e^{-\lambda t_0})>0,
\end{align*}
with $\eta$ defined by \eqref{e:q2}, which gives \ref{cnd:b3}.

Now, we shall establish condition \ref{cnd:b4}. To do this, fix $(x_1,x_2):=((y_1,i_1),(y_2,i_2))\in F$, and note that, due to \eqref{e:q3}, 
\begin{align}
\begin{split}
\label{e:q_full}
Q_P((x_1,x_2),X^2)&=\sum_{j\in I} (\pi_{i_1,j}\wedge\pi_{i_2,j})\int_0^{\infty} \lambda e^{-\lambda h} Q_J((S_{i_1}(h,y_1),S_{i_2}(h,y_2)),Y^2)\,dh\\
&\geq \sum_{j\in I} (\pi_{i_1,j}\wedge\pi_{i_2,j})-\tilde{l}\lambda\int_0^{\infty} e^{-\lambda h}\rho_Y(S_{i_1}(h,y_1),S_{i_2}(h,y_2))\,dh.
\end{split}
\end{align}
On other hand, referring again to \eqref{e:est_s}, we get
\begin{align}
\begin{split}
\label{e:q_int}
&\int_0^{\infty}e^{-\lambda h} \rho_Y(S_{i_1}(h,y_1),S_{i_2}(h,y_2))\,dh
\\
&\leq  L\left(\int_0^{\infty} e^{-(\lambda-\alpha)h}\,dh\right)\rho_Y(y_1,y_2)+ M_{\mathcal{L}}\left(\int_0^{\infty} e^{-\lambda h}\varphi(h)\,dh\right)\dd(i_1,i_2)\\
&\leq \frac{L}{\lambda-\alpha}\rho_Y(y_1,y_2)+M_{\mathcal{L}}K_{\varphi}\dd(i_1,i_2).
\end{split}
\end{align}
Moreover, we can write
\begin{equation}
\label{e:q_sum}
\sum_{j\in I} \pi_{i_1,j}\wedge \pi_{i_2,j}\geq 1-\,\dd(i_1,i_2).
\end{equation}
Consequently, taking into account \eqref{e:q_full}, \eqref{e:q_int}, \eqref{e:q_sum} and \eqref{e:const_c}, we infer that
\begin{align*}
Q_P((x_1,x_2),X^2)&\geq 1-\frac{\tilde{l}\lambda L}{\lambda-\alpha}\rho_Y(y_1,y_2)-(1+\tilde{l}\lambda M_{\mathcal{L}}K_{\varphi})\dd(i_1,i_2)\\
&\geq 1-\left(\frac{\tilde{l}\lambda L}{\lambda-\alpha}+1+\tilde{l}\lambda M_{\mathcal{L}}K_{\varphi}\right)\rho_{X,c}(x_1,x_2),
\end{align*}
which completes the proof of \ref{cnd:b4}.

What remains is to show that \ref{cnd:b5} holds. To do this, we shall apply \hbox{\cite[Lemma 2.2]{b:kapica}} to the coupling $\{(\Phi_n^{(1)},\Phi_n^{(2)})\}_{n\in\n_0}$ with transition law $\widetilde{P}$, given by \eqref{def:P_tilde}. For~this purpose, it suffices to observe that, letting
$$\widetilde{V}(y_1,y_2):=V(y_1)+V(y_2)=\rho_Y(y_1,y^*)+\rho_Y(y_2,y^*)\quad\text{for}\quad (y_1,y_2)\in X^2,$$
we have
$$\widetilde{P}\widetilde{V}(x_1,x_2)\leq a \widetilde{V}(x_1,x_2)+2b\quad\text{for any}\quad(x_1,x_2)\in X^2,$$
which follows directly from \ref{cnd:b1} and \eqref{e:est_s}. The proof of Lemma \ref{lem:help} is now complete.
\end{proof}
Having established Lemma \ref{lem:help}, we can prove Proposition \ref{lem:coup} by arguing as in the proof of \cite[Lemma 2.1]{b:czapla_clt}. First of all, we need to be able to distinguish the
case where the next step of the chain $\widehat{\Phi}$ is drawn only according to $\bar{Q}_P$ from the case when it is determined only by $\bar{R}_P$. For this aim, we consider $\widehat{Z}:=Z\times\{0,1\}$, which can be viewed as a copy of $Z=X^2\times\mathbb{R}_+$, splitted into two disjoint subsets $Z\times\{0\}$ and $Z\times\{1\}$. Then we define a new stochastic kernel $\ps:\widehat{Z}\times \mathcal{B}(\widehat{Z})\to [0,1]$ by setting
$$\ps((x_1,x_2,s,k),H)=(\bar{Q}_P((x_1,x_2,s),\cdot)\otimes\delta_1^{\star})(H)+(\bar{R}_P((x_1,x_2,s),\cdot)\otimes\delta_0^{\star})(H)$$
for any $x_1,x_2\in X$, $s\in\mathbb{R}_+$, $k\in\{0,1\}$ and $H\in\mathcal{B}(\widehat{Z})$, where $\delta_0^{\star}$ (resp. $\delta_1^{\star}$) stands for the Dirac measure at $0$ (resp. at $1$) on $2^{\{0,1\}}$. Obviously, for any $(x_1,x_2,s,k)\in\widehat{Z}$ and $D\in\mathcal{B}(Z)$, we have
\begin{gather*}
\ps((x_1,x_2,s,k),D\times \{1\})=\bar{Q}_P((x_1,x_2,s),D),\\
\ps((x_1,x_2,s,k),D\times \{0\})=\bar{R}_P((x_1,x_2,s),D),\\ 
\ps((x_1,x_2,s,k),D\times \{0,1\})=\widehat{P}((x_1,x_2,s),D).
\end{gather*}

Further, we introduce the canonical Markov chain $\cs:=\{(\widehat{\Phi}_n',\kappa_n)\}_{n\in\n_0}$ with transition law~$\ps$, wherein $\kappa_n$ takes values in $\{0,1\}$, and $\widehat{\Phi}'=\{(\Phi_n'^{(1)},\Phi_n'^{(2)},\tau_n')\}_{n\in\n_0}$ is an appropriate copy of $\widehat{\Phi}$. We therefore assume that $\cs$ is defined on the space $(\widehat{\Omega}^{\star},\widehat{\mathcal{F}}^{\star}):=(\widehat{Z}^{\mathbb{N}_0},\mathcal{B}({\widehat{Z}}^{\mathbb{N}_0}))$, equipped with an appropriate family $\{\prs_{(x_1,x_2)}:\,x_1,x_2\in X\}$ of probability measures on $\widehat{\mathcal{F}}^{\star}$, such that $\cs$ starts at $((x_1,x_2,0),0)$ almost surely with respect to $\prs_{(x_1,x_2)}$. The symbol $\ers_{(x_1,x_2)}$ will denote the expectation operator corresponding to $\prs_{(x_1,x_2)}$.

Let us now fix arbitrarily $(x_1,x_2)\in X^2$ and $n,M,N\in\n$ such that $n>M>N$. Further, consider the random times
\begin{gather*}
\sigma_K^{(N)}:=\inf\left\{m\geq N:\,\left(\Phi_m^{'(1)},\Phi_m^{'(2)} \right)\in K\right\}, \quad \zeta:=\inf\{m\in\n:\,\kappa_i=1\;\;\text{for any}\;\;i\geq m\},
\end{gather*}
where $K$ is given by \eqref{def:F}, and define
$$H_{N,n}:=\{\kappa_N=\kappa_{N+1}=\ldots=\kappa_n=1\},\quad H_{N,n}^c:=\widehat{\Omega}^{\star}\backslash H_{N,n}.$$

Taking into account that $\prs(H_{N,n}^c)\leq \prs(\zeta>N)$, and that $\bar{\rho}_{X,c}(y_1,y_2)\leq 1$ for any $y_1,y_2\in Y$, we can write the following estimate:
\begin{align*}
\widehat{\ew}_{(x_1,x_2)}\left[\bar{\rho}_{X,c}\left(\Phi_n^{(1)}, \Phi_n^{(2)}\right)\right]&=\ers_{(x_1,x_2)}\left[\bar{\rho}_{X,c}\left(\Phi_n'^{(1)}, \Phi_n'^{(2)} \right) \right]\\
&=\int_{X^2} \bar{\rho}_{X,c}\left(y_1,y_2\right)\prs_{(x_1,x_2)}\left((\Phi_n'^{(1)},\Phi_n'^{(2)})\in dy_1\times dy_2\right)\\
&\leq \int_{X^2} \bar{\rho}_{X,c}\left(y_1,y_2\right)\prs_{(x_1,x_2)}|_{\left\{\sigma_K^{(N)}\leq M\right\}\cap H_{N,n}}\left((\Phi_n'^{(1)},\Phi_n'^{(2)})\in dy_1\times dy_2\right)\\
&\quad+\prs_{(x_1,x_2)}(\sigma_K^{(N)}>M)+\prs_{(x_1,x_2)}(\zeta>N),
\end{align*}
with the convention that $\prs_{(x_1,x_2)}|_H(\cdot):=\prs_{(x_1,x_2)}(H\cap \cdot)$.

Since, according to Lemma \ref{lem:help}, hypotheses \ref{cnd:b1}-\ref{cnd:b5} are fulfilled, we can now apply \hbox{\cite[Lemma 2.2]{b:czapla_clt}} to conclude that there exist constants $C_1,C_2,C_3\geq 0$, $q_1,q_2,q_3\in (0,1)$ and $p\geq 1$ such that
\begin{align*}
\widehat{\ew}_{(x_1,x_2)}\left[\bar{\rho}_{X,c}\left(\Phi_n^{(1)}, \Phi_n^{(2)}\right)\right]\leq (C_1 q_1^{n-M}+C_2 q_2^{M-pN}+C_3q_3^N)(1+V(x_1)+V(x_2)).
\end{align*}

Finally, letting $n\geq \lceil 4p \rceil$ and taking $N=\lfloor n/(4p) \rfloor$, $M=\lfloor n/2 \rfloor$, we obtain
$$\widehat{\ew}_{(x_1,x_2)}\left[\bar{\rho}_{X,c}\left(\Phi_n^{(1)},\Phi_n^{(2)}\right)\right]\leq \widetilde{C}_0(V(x_1)+V(x_2)+1)q^n$$
with $q:=\{q_1^{1/2},q_2^{1/4},q_3^{1/(4p)}\}$ and $\widetilde{C}_0=\max\{q_1^{-1},q_3^{-1}\}(C_1+2C_2+2C_3)$. Obviously, since $\bar{\rho}_{X,c}\leq 1$, this inequality holds, in fact, for all $n\in\n$ with $C_0:=q^{-\lceil 4p \rceil}\max\{\widetilde{C}_0,1\}$ in the place of $\widetilde{C}_0$. The proof of Proposition~\ref{lem:coup} is now complete.

\section*{Acknowledgements}
The work of Hanna Wojew\'odka-\'Sci\k{a}\.zko has been partly supported by the National Science Centre of Poland, grant number 2018/02/X/ST1/01518.

\bibliographystyle{plain}
\bibliography{ReferencesDatabase}
\end{document}